\documentclass[a4paper,11pt]{amsart}
\usepackage{color}
\usepackage{amsmath,amscd,amsthm,upref}
\usepackage[psamsfonts]{amssymb}
\usepackage[all]{xy}
\usepackage{enumerate}

\newtheorem{thm}{Theorem}[section]
\newtheorem{crl}[thm]{Corollary}
\newtheorem{lmm}[thm]{Lemma}
\newtheorem{prp}[thm]{Proposition}
\theoremstyle{definition}
\newtheorem{dfn}[thm]{Definition}

\theoremstyle{remark}
\newtheorem*{rmk}{Remark}

\DeclareMathOperator{\Diff}{\bf Diff}
\DeclareMathOperator{\Top}{\bf Top}

\newcommand{\bI}{\partial I}
\newcommand{\abs}[1]{\lvert{#1}\rvert}

\title{Homotopy structures of smooth CW complexes}
\author[T. Haraguchi]{Tadayuki Haraguchi} 
\address{Faculty of Education for Human Growth, Nara Gakuen University, Nara 636-8503, Japan}
\email{t-haraguchi@naragakuen-u.jp}
\date{\today}
\subjclass[2010]{Primary 18G55 ; Secondary 55P05, 55P10, 55Q05}
\keywords{diffeological space, smooth CW complex, homotopy extension property, cellular approximation theorem, Whitehead theorem}
\begin{document}
\maketitle
\begin{abstract}
In this paper we present the notion of smooth CW complexes given by attaching cubes on the category of diffeological spaces,
and we study their smooth homotopy structures related to the  homotopy extension property.
\end{abstract}
\tableofcontents
%
%
%
%
%
%
%
%
\section{Introduction}
The homotopy theory has a long history.
In topology,
a CW complex is a type of topological space introduced by J.~H.~C Whitehead to meet the needs of homotopy theory (cf.~\cite{JHC0}).
To develop homotopy theory,
Quillen gives,
so called,
the Quillen model structure on the category $\bf Top$ of topological spaces which represents 
the standard homotopy theory of CW complexes (cf.~\cite{Qui67}).
Model categories have been constructed in many fields to solve various ploblems having homotopical means.
We present the Quillen model structure on the category $\Diff$ of diffeological spaces which is Quillen equivalent to 
the model structure above on $\Top$ by constructing smooth homotopy theory on homotopy groups and cubical cell complexes 
(cf.~{\cite[Theorem 5.6 and Theorem 6.2]{HS}}).
Cubical cell complexes are cofibrant objects which represent the behavior of our model structure on $\Diff$ (cf.~{\cite[\S6]{HS}}).
In this article,
we give smooth CW complexes as a class of cubical cell complexes and introduce their smooth homotopy theory related to the homotopy extension property.
%

In topology,
the property of retractions from the cube $I^{n}$ to the horn $J^{n-1}$ develops the homotopy theory.
However,
in diffeology,
due to the difficulty finding a smooth retraction $I^{n} \to J^{n-1}$,
we must retrieve most of the ingredients of homotopy theory without relying such a strict retraction.
To solve this ploblem,
In \cite[\S3]{HS},
we introduce the notion of \textit{tameness} into smooth maps from $I^{n}$ to a diffeological space $X$.
A tame map $I^{n} \to X$ is locally constant at all vertexes (see Definition \ref{dfn:definition of tame maps}).
By approximate retractions $I^{n} \to J^{n-1}$,
any tame map $J^{n-1} \to X$ can be extended to a smooth map $I^{n} \to X$  (see Lemma \ref{lmm:approximate retraction exists} and Proposition \ref{extend of J}).
Using this property,
we discuss our homotopy structures on smooth CW complexes.

This paper is organized as follows.
Section 2 review the basic properties of tameness,
approximate retractions and smooth homotopy groups.
Section 3 introduces the \textit{final diffeology} defined by functions.
The colimits in $\Diff$ are always equipped with the final diffeology defined by natural maps (cf.~\cite{DGE}). 
Here we study the final diffeology of the product of pushouts.
In section 4,
we give smooth CW complexes (see Definition \ref{dfn:definition of smooth CW complexes})
equipped with the final diffeology defined by characteristic maps of the cubes,
and we study their basic structures (see Proposition \ref{prp:product of CW},
\ref{prp:coproduct of CW complexes} and \ref{prp:quotient of CW complexes}).
Furthermore we introduce smooth CW complexes equipped with tameness to discuss our smooth homotopy theory.
Such smooth CW complexes are called \textit{gathered}.
Any diffeological space has the natural topology called the $D$-topology (cf.~\cite{Zem85}).
Our interest shifts to the behavior of the $D$-topology for smooth CW complexes.
The topological spaces which arise as the $D$-topology of smooth CW complexes are 
paracompactum CW complexes equipped with the weak topology (see Proposition \ref{topological CW complexes} and \ref{the weak topology induced smooth CW complexes}).
Furthermore we have the Whitney approximation property on gathered CW complexe,
that is,
any continuous map between gathererd CW complexes are continuously homotopic to some smooth map (see Theorem \ref{thm:Whitney approximation property}).
In Section 5 we give the homotopy extension property (see Lemma \ref{lmm:HEP}),
and develop this property further by introducing the notions of $N$-\textit{connected} and $N$-\textit{equivalence}
(Definition \ref{dfn:N-connected} and \ref{dfn:weak homotopy equivalence}).
As their applications,
in Section 6-7 we prove the following main results.
\begin{enumerate}
\item
A relative gathered CW complex has the homotopy extension lifting property (see Theorem \ref{HELP}).
\item
Let $\phi \colon (X;A_{1},A_{2}) \to (Y;B_{1},B_{2})$ be a smooth map of excive triads such that the restrictions 
\[
\phi \colon A_{1} \to B_{1},\ \phi \colon A_{2} \to B_{2} \ \mbox{and} \ \phi \colon A_{1} \cap A_{2} \to B_{1} \cap B_{2}
\]
are weak homotopy equivalences.
Then $\phi \colon X \to Y$ is a weak homotopy equivalence (see Theorem \ref{thm:extends weak}).
\item
We have the Whitehead theorem on gathered CW complexes,
that is,
a weak homotopy equivalence between gathered CW complexes is a homotopy equivalence (see Teorem \ref{thm:Whitehead}).
\item
A smooth map between gathered CW complexes is homotopic to some cellular map (see Theorem \ref{thm:cellular approximation}).
\end{enumerate}
%
%
%
%
\section{Preliminaries}\label{the basic notion of diffeology}
In this section we introduce the notions of \emph{tameness} and \emph{homotopy groups}
to study a smooth homotopy 
theory of diffeological spaces.
All the material in this section are described in \cite[Section 3]{HS}.
The fundamental properties of diffeological spaces can be found in the standard textbook \cite{Zem}.
The class of diffeological spaces together with smooth maps form a category $\bf Diff$ which is complete,
cocomplete and cartesian closed (cf.~\cite{Sou80}).

Throughout the paper,
Let $\mathbf{R}^{n}$ be the $n$-dimensional Euclidean space equipped with the standard diffeology consisting of all smooth parametrizations of
$\mathbf{R}^{n}$,
and let $I^{n}$ be the $n$-dimensional cube $I^{n}$ equipped with the subspace diffeology.
Denote by $\partial I^{n}$ the boundary of $I^{n}$,
and let
\[
L^{n-1}= \partial I^{n-1} \times I \cup I^{n-1} \times \{0\}\ \mbox{and}\
J^{n-1}= \partial I^{n-1} \times I \cup I^{n-1} \times \{1\}
\]
for each $n \geq 1$.
We regard $\partial I^{n},\ L^{n-1}$ and $J^{n-1}$ as subspaces of $\mathbf{R}^{n}$.
Suppose $f_{0},f_{1} \colon X \to Y$ are smooth maps between diffeological spaces.
We say that $f_{0}$ and $f_{1}$ are homotopic,
written $f_{0} \simeq f_{1}$,
if there exists a smooth map $F \colon X \times I \to Y$ such that $F(x,0)=f_{0}(x)$ and $F(x,1)=f_{1}(x)$ hold.
Such a smooth map $F$ is called a homotopy between $f_{0}$ and $f_{1}$.
Homotopy is an equivalence relation.
The resulting equivalent classes are called homotopy classes.
Moreover,
$f_{0}$ and $f_{1}$ are called homotopic relative to a subspace $A$ of $X$,
written $f_{0} \simeq f_{1} \ \mbox{rel} \ A$,
if there exists a homotopy $F \colon X \times I \to Y$ between
$f_{0}$ and $f_{1}$ such that $F(a,t)=f_{0}(a)=f_{1}(a)$ holods for all $a \in A$ and for all $t \in I$.
A smooth map $f \colon X \to Y$ is called homotopy equivalence if there exists a smooth map $g \colon Y \to X$ satisfying
\[
g \circ f \simeq 1 \colon X \to X,\ f \circ g \simeq 1_{Y} \colon Y \to Y.
\]
We say that $X$ and $Y$ are homotopy equivalent,
written $X \simeq Y$,
if there exists a homotopy equivalence $f \colon X \to Y$.
A subspace $A$ is called a retract of $X$ 
if there exists a smooth map $\gamma \colon X \to A$ such that $\gamma(a)=a$ for all $a \in A$.
Such a smooth map $\gamma \colon X \to A$ is called a retraction.
We say that $A$ is a deformation retract of $X$ if there exists a retraction $\gamma \colon X \to A$ such that the composition $i \circ \gamma$ and the identity $1_{X} \colon X \to X$
are homotopic relative to $A$,
where $i \colon A \to X$ is the inclusion.

In the case of topological spaces,
the fact that $J^{n-1}$ is a retract of $I^{n}$ is crucial for developing homotopy theory.
(Compare e.g. homotopy exact sequence and homotopy extension property.)
Unfortunately,
it is not easy to construct a smooth retraction from $I^{n}$ to $J^{n-1}$.
Thus
we discuss our smooth homotopy theory by using the following smooth map which is locally constant at each vertex.
\begin{dfn}[{\cite[Definition 3.8]{HS}}]\label{dfn:definition of tame maps}
Let $f \colon K \to X$ be a smooth map from a cubical subcomplex $K$
of $I^n$ (e.g.\ $I^n$, $\bI^n$, $L^{n-1}$ or $J^{n-1}$) to a diffeological
space $X$.  Suppose $0 < \epsilon \leq 1/2$ and $\alpha \in \{0,1\}$.
Then $f$ is called to
be \emph{$\epsilon$-tame} if we have
\[
f(t_1,\cdots,t_{j-1},t_j,t_{j+1},\cdots,t_n) =
f(t_1,\cdots,t_{j-1},\alpha,t_{j+1},\cdots,t_n)
\]
for every $(t_1,\cdots,t_n) \in K$ such that $\abs{t_j - \alpha} \leq \epsilon$ holds.
Moreover,
a homotopy $F \colon X \times I \to Y$ is called to be $\epsilon$-tame if we have $F(x,t)=F(x,\alpha)$ for all $x \in X,\ t\in I$ such that
$| t - \alpha | \leq \epsilon $ holods.
We use the abbreviation ``tame'' to mean $\epsilon$-tame for some
$\epsilon > 0$.
\end{dfn}
We shall introduce that any tame map defined on $J^{n-1}$ is extendable over $I^{n}$.
To see this we need some preparations.
For $0 < \epsilon \leq 1/2$, we denote
$I^n(\epsilon) = [\epsilon,1-\epsilon]^n$ and call it the
\emph{$\epsilon$-chamber} of $I^n$.
More generally, if $K$ is a cubical subcomplex of $I^n$ then its
$\epsilon$-chamber $K(\epsilon)$ is defined to be the union of
$\epsilon$-chambers of its maximal faces.
Thus we have $\bI^n(\epsilon) = \bigcup F(\epsilon)$, where $F$ runs
through the $(n-1)$-dimensional faces of $I^n$, and
$J^{n-1}(\epsilon) = \bI^n(\epsilon) \cap J^{n-1}$.
\begin{dfn}
A smooth map $I^{n} \to J^{n-1}$ is called an $\epsilon$-\emph{approximate retraction} 
if it restricts to the identity on the $\epsilon$-chanber $J^{n-1}(\epsilon)$.
\end{dfn}
To get an approximate retraction,
we construct the following tame map.
\begin{lmm}[{\cite[Lemma 3.10]{HS}}]
  \label{lmm:modified smash function}
  Suppose $0 \leq \sigma < \tau \leq 1/2$.  Then there exists a
  non-decreasing smooth function $T_{\sigma,\tau} \colon \mathbf{R} \to I$
  satisfying the following conditions:
  \begin{enumerate}
  \item $T_{\sigma,\tau}(t) = 0$ for $t \leq \sigma$,
  \item $T_{\sigma,\tau}(t) = t$ for $\tau \leq t \leq 1 - \tau$,
  \item $T_{\sigma,\tau}(t) = 1$ for $1 - \sigma \leq t$, and
  \item $T_{\sigma,\tau}(1 - t) = 1 - T_{\sigma,\tau}(t)$
    for all $t$.
  \end{enumerate}
\end{lmm}
Let $T^{n}_{\sigma, \tau} \colon \mathbf{R}^{n} \to I^{n}$ denote the $\sigma$-tame map given by
\[
T^{n}_{\sigma, \tau}(t_{1}, \cdots, t_{n})=\left(T_{\sigma, \tau}(t_{1}), \cdots, T_{\sigma, \tau}(t_{n})\right).
\]
\begin{lmm}[{\cite[Lemma 3.12]{HS}}]
  \label{lmm:approximate retraction exists}
  For any real number $\epsilon$ with $0 < \epsilon < 1/2$, there
  exists an $\epsilon$-approximate retraction
  $R_{\epsilon}^{n} \colon I^n \to J^{n-1}$.
\end{lmm}
It is clear that we have the following by the lemmas above.
\begin{prp}[{\cite[Proposition 3.13]{HS}}]\label{extend of J}
Any $\epsilon$-tame map $f \colon J^{n-1} \to X$ can be extended to a $\sigma$-tame map $g \colon I^{n} \to X$ for any 
$\sigma < \epsilon$.
\end{prp}
\begin{proof}
Let $f \colon J^{n-1} \to X$ be an $\epsilon$-tame map.
Then we have a $\sigma$-tame map $g$ given by the composition $f \circ R^{n}_{\epsilon} \circ T^{n}_{\sigma, \epsilon} \colon I^{n} \to X$.
\end{proof}
Next,
we will intoroduce homotopy groups of diffeological spaces.
Let $f_{0},f_{1} \colon (X,A) \to (Y,B)$ be smooth maps between pairs of diffeological spaces,
then we define $f_{0}$ and $f_{1}$ to be homotopic,
written $f_{0} \simeq f_{1}$,
if there is a homotopy $H$ between $f_{0}$ and $f_{1}$ such that $H(a,t) \in B$ for all $t \in I$.
Such $H$ is called a homotopy of pairs.
A smooth map of pairs $f \colon (X,A) \to (Y,B)$ is called a homotopy equivalence if there exists a smooth map of pairs $g \colon (Y,B) \to (X,A)$ satisfying $g \circ f \simeq 1_{X}$ and $f \circ g \simeq 1_{Y}$.
We denote by
\[
[X,A;Y,B]
\]
the set of homotopy classes of smooth maps between pairs $(X,A)$ and $(Y,B)$.
Similarly,
let us denote by 
\[
[X,A_{1},A_{2}; Y, B_{1}, B_{2}]
\]
the set of homotopy classes of smooth maps between triples $(X,A_{1},A_{2})$ and $(Y,B_{1},B_{2})$.
\begin{dfn}
Given a pointed diffeological space $(X,x_{0})$,
we put 
\[
\pi_{n}(X,x_{0})=[I^{n}, \partial I^{n} ; X,x_{0}],\ n \geq 0.
\]
Similarly,
given a pointed pair of diffeological spaces $(X,A,x_{0})$,
we put 
\[
\pi_{n}(X,A,x_{0})=[I^{n}, \partial I^{n}, J^{n-1} ; X,A,x_{0}], \ n \geq 1.
\]
\end{dfn}
For $n \geq 1$, $\pi_n(X,x_0)$ is isomorphic to $\pi_n(X,x_0,x_0)$,
and $\pi_0(X,x_0)$ is isomorphic to the set of path components
$\pi_0{X}$, regardless of the choice of basepoint $x_0$.  Note,
however, that we consider $\pi_0(X,x_0)$ as a pointed set with
basepoint $[x_0] \in \pi_0{X}$.
We now introduce a group structure on $\pi_n(X,A,x_0)$.
Suppose
$\phi$ and $\psi$ are smooth maps from $(I^n,\bI^n,J^{n-1})$ to
$(X,A,x_0)$.  If $n \geq 2$, or if $n \geq 1$ and $A = x_0$, then
there is a smooth map $\phi * \psi \colon I^n \to X$ which takes
$(t_1,t_2,\dots,t_n) \in I^n$ to
\[
  \begin{cases}
    \phi(T_{\sigma,\tau}(2t_1),t_2,\dots,t_n), & 0 \leq t_1 \leq 1/2
    \\
    \psi(T_{\sigma,\tau}(2t_1-1),t_2,\dots,t_n), & 1/2 \leq t_1 \leq 1.
  \end{cases}
\]
It is clear that $\phi * \psi$ defines a map of triples
$(I^n,\bI^n,J^{n-1}) \to (X,A,x_0)$, and there is a multiplication on
$\pi_n(X,A,x_0)$ given by the formula
\[
  [\phi] \cdot [\psi] = [\phi * \psi] \in \pi_n(X,A,x_0).
\]
Then the homotopy set $\pi_{n}(X,A,x_{0})$ is a group if $n \geq 2$ or if $n \geq 1$ and $A=x_{0}$,
and is an abelian group if $n \geq 3$ or if $n \geq 2$ and $A=x_{0}$.
In general,
the homotopy sets $\pi_{n}(X,x_{0})$ and $\pi_{n}(X,A,x_{0})$ are called homotopy group and relative homotopy group,
respectively.
Moreover,
for every smooth map between pointed pairs $f \colon (X,A,x_{0}) \to (Y,B,y_{0})$,
the induced map $f_{\ast} \colon \pi_{n}(X,A,x_{0}) \to \pi_{n}(Y,B,y_{0})$ is a group homomorphism whenever its source and target are groups.
In particular,
any element of $\pi_{n}(X,A,x_{0})$ is represented by a tame map $(I^{n}, \partial I^{n},J^{n-1}) \to (X,A,x_{0})$ by {\cite[Lemma 3.11]{HS}}.
%
%
%
%
%
%
\section{The final diffeology}
In this section we will introduce the properties of the \emph{final diffeology} (cf.~{\cite[p.3]{DGE}}) 
to make homotopy structures of smooth CW complexes easy to discuss.
In particular, the pushouts are equipped with the final diffeology.
%

Let $Z$ be a set.
Let us denote by $\{(X_{\lambda},f_{\lambda}) \}_{\lambda \in \Lambda}$  the family of pairs of a diffeological space $X_{\lambda}$ 
and a function $f_{\lambda} \colon X_{\lambda} \to Z$.
We say that $\mathcal{F}=\{(X_{\lambda},f_{\lambda}) \}_{\lambda \in \Lambda}$ is a \emph{functionalized cover} of $Z$ if the union of the image of $f_{\lambda}$ covers $Z$,
that is,
$\cup_{\lambda \in \Lambda} f_{\lambda}(X_{\lambda})=Z$ holds.
Then there exists a finest diffeology of $Z$ such that all functions $f_{\lambda}$ are smooth,
and it will be called the \emph{final diffeology} defined by $\mathcal{F}$ and will be denoted by $<\mathcal{F}>$.
Then a parametrization $P \colon U \to Z$ is a plot of $Z$ if and only if it is lifts locally,
at each $r \in U$,
along a function $f_{\lambda} \colon X_{\lambda} \to Z$ (cf.~{\cite[1.68]{Zem}}).
Conversly,
let $Z$ be a diffeological space and $D$ be its diffeology.
We say that $D$ is \emph{generated} by $\mathcal{F}$,
if $D=< \mathcal{F} >$ holds.
In particular,
the final diffeology is a generalization of the weak diffeology (cf.~{\cite[p.17]{HS}}).
For example,
let $\mathcal{I}=\{(X_{\lambda}, i_{\lambda})\}_{\lambda \in \Lambda}$ be a functionalized cover of 
a diffeological space $X$ consisting of the inclusions $i_{\lambda}$ from subspace $X_{\lambda}$ to $X$.
Then the final diffeology defined by $\mathcal{I}$ is the weak diffeology.
We shall present the fundamental properties of the final diffeologies.
\begin{prp}
Let $Z$ be a diffeological space equipped with the final diffeology defined by 
its functionalized cover $\{(X_{\lambda}, f_{\lambda})\}_{\lambda \in \Lambda}$,
and let $Z^{\prime}$ be a diffeological space.
Then a map $g \colon Z \to Z^{\prime}$ is a smooth if and only if the composition 
$g \circ f_{\lambda} \colon X_{\lambda} \to Z^{\prime}$ is smooth for each $\lambda \in \Lambda$.
\end{prp}
\begin{proof}
Let $P \colon U \to Z$ be a plot of $Z$.
For any $r \in U$,
there exist an open neighborhood $V$ of $r$ and a plot $Q_{\lambda} \colon V \to X_{\lambda}$ of $X_{\lambda}$ satisfying $P|V=f_{\lambda} \circ Q_{\lambda}$.
Then
$
g \circ P|V= g \circ (f_{\lambda} \circ Q_{\lambda})
$
is a plot of $Z^{\prime}$.
Thus $g$ is smooth.
\end{proof}
\begin{prp}\label{product of final diffeology}
Let $Z$ and $Z^{\prime}$ be a diffeological spaces equipped with the final diffeologies defined by 
their functionalized covers $\{(X_{\lambda}, f_{\lambda})\}_{\lambda \in \Lambda}$ and 
$\{(Y_{\lambda^{\prime}}, g_{\lambda^{\prime}}) \}_{\lambda^{\prime} \in \Lambda^{\prime}}$,
respectively.
Then the product diffeology of $Z \times Z^{\prime}$ is generated by
$\{(X_{\lambda} \times Y_{\lambda^{\prime}},f_{\lambda}\times g_{\lambda^{\prime}}) \}$.
\end{prp}
\begin{proof}
Let $P \colon U \to Z \times Z^{\prime}$ be a plot of $Z \times Z^{\prime}$,
where $P(r)=(P_{Z}(r), P_{Z^{\prime}}(r))$.
Then for any $r \in U$,
there are a plot $Q_{\lambda} \colon V \to X_{\lambda}$ of $X_{\lambda}$ and a plot 
$Q^{\prime}_{\lambda^{\prime}} \colon V^{\prime} \to Y_{\lambda^{\prime}}$ of $Y_{\lambda^{\prime}}$ such that
$P_{Z}|V=f_{\lambda} \circ Q_{\lambda}$ and $P_{Z^{\prime}}|V^{\prime}=g_{\lambda^{\prime}} \circ Q^{\prime}_{\lambda^{\prime}}$,
respectively.
Let us define the plot 
$Q_{(\lambda, \lambda^{\prime})} \colon V \cap V^{\prime} \to X_{\lambda} \times Y_{\lambda^{\prime}}$ by 
$Q_{(\lambda,\lambda^{\prime})}(s)=(Q_{\lambda}(s),Q^{\prime}_{\lambda^{\prime}}(s))$.
Then we have $P|V \cap V^{\prime}=(f_{\lambda} \times g_{\lambda^{\prime}}) \circ Q_{(\lambda,\lambda^{\prime})}$.
\end{proof}
Clearly,
we have the following by Proposition \ref{product of final diffeology}.
\begin{crl}\label{final diffeology and identity}
Let $Z$ be a diffeological space equipped with the final diffeology defined by its functionalized cover 
$\{(X_{\lambda},f_{\lambda}) \}_{\lambda \in \Lambda}$.
Then for any diffeological space $Z^{\prime}$,
the product diffeology of $Z \times Z^{\prime}$ is generated by $\{(X_{\lambda} \times Z^{\prime},f_{\lambda} \times 1_{Z^{\prime}})\}$,
where $1_{Z^{\prime}}$ is the identity $Z^{\prime} \to Z^{\prime}$.
\end{crl}
Let $\{f_{\lambda} \}_{\lambda \in \Lambda}$ be the family of smooth maps $f_{\lambda} \colon X_{\lambda} \to Z$.
We denote by $\bigcup_{\lambda \in \Lambda}f_{\lambda}$ the composition
\[
\bigtriangledown \circ ( \coprod_{\lambda \in \Lambda} f_{\lambda} ) \colon 
\coprod_{\lambda \in \Lambda} X_{\lambda} \to \coprod_{\lambda \in \Lambda} Z \to Z,
\]
where $\bigtriangledown$ is the foloding map of $\coprod_{\lambda \in \Lambda} Z$ onto $Z$.
Then we have the following.
\begin{lmm}\label{subduction of final diffeology}
Let $Z$ be a diffeological space and 
$\{(X_{\lambda},f_{\lambda}) \}_{\lambda \in \Lambda}$
be its functionalized cover.
Then the following conditions are equivalent.
\begin{enumerate}
\item
The diffeology of $Z$ is generated by $\{(X_{\lambda},f_{\lambda}) \}_{\lambda \in \Lambda}$.
\item
$
\bigcup_{\lambda \in \Lambda} f_{\lambda} \colon \coprod_{\lambda \in \Lambda} X_{\lambda} \to Z
$
is a subduction.
\end{enumerate}
Moreover, 
let $Z^{\prime}$ be a diffeological space.
Then the following conditions are equivalent.
\begin{enumerate}
\item
The product diffeology of $Z \times Z^{\prime}$ is generated by $\{(X_{\lambda} \times Z^{\prime}, f_{\lambda} \times 1_{Z^{\prime}})\}$.
\item
$
\bigcup_{\lambda \in \Lambda}(f_{\lambda} \times 1_{Z^{\prime}}) \colon \coprod_{\lambda \in \Lambda}(X_{\lambda} \times Z^{\prime})
\to Z \times Z^{\prime}
$
is a subduction.
\end{enumerate}
\end{lmm}
\begin{proof}
It is clear that $\bigcup_{\lambda \in \Lambda}f_{\lambda}$ is a subduction by the definition of the final diffeology.
Conversely,
any $P \colon U \to Z$ be a plot of $Z$ is locally lifts,
at each $r \in U$,
along a smooth map $f_{\lambda}$ by the definitions of the subduction and the coproduct diffeology.
That the next two conditions are equivalent is immediate from Corollary \ref{final diffeology and identity}.
\end{proof}
Given the following pushout square
\[
\xymatrix{
A
\ar[r]^{i_{1}}
\ar[d]_{i_{2}}
&
X_{1}
\ar[d]^{f_{1}}
\\
X_{2}
\ar[r]_(0.3){f_{2}}
&
X_{1} \cup_{(i_{1}, i_{2})} X_{2}.
}
\]
Then the pushout $X_{1} \cup_{(i_{1},i_{2})} X_{2}$ is the quotient space $X_{1} \coprod X_{2} /\sim$,
with the quotient map $\pi \colon X_{1} \coprod X_{2} \to X_{1} \cup_{(i_{1},i_{2})} X_{2}$ satisfying 
$\pi|X_{1}=f_{1}$ and $\pi|X_{2}=f_{2}$,
where $\sim$ is the smallest equivalence relation on the coproduct space $X_{1} \coprod X_{2}$ that $i_{1}(a) \sim i_{2}(a)$ for all $a \in A$.
Since $\pi=f_{1} \bigcup f_{2}$ is the subduction by Lemma \ref{subduction of final diffeology},
its diffeology is generated by $\{(X_{j},f_{j})\}_{j=1,2}$.
When $A$ is a subspace of $X_{2}$,
the pushout $X_{1} \cup_{i_{1}} X_{2}$ is called the adjunction space of $X_{2}$ to $X_{1}$ along $i_{1}$ and its diffeology is called the 
\emph{gluing diffeology} (cf.~\cite{Pervova}).
Now,
we have the following.
\begin{prp}
Let $Z^{\prime}$ be a diffeological space.
Then the following commutative diagram
\[
\xymatrix{
A \times Z^{\prime}
\ar[r]^{i_{1} \times 1_{Z^{\prime}} }
\ar[d]_{i_{2} \times 1_{Z^{\prime}} }
&
X_{1} \times Z^{\prime}
\ar[d]^{f_{1} \times 1_{Z^{\prime}}}
\\
X_{2} \times Z^{\prime}
\ar[r]_(0.35){f_{2} \times 1_{Z^{\prime}} }
&
(X_{1} \cup_{(i_{1}, i_{2})} X_{2}) \times Z^{\prime}
}
\]
is a pushout.
\end{prp}
\begin{proof}
By Corollary \ref{final diffeology and identity} and Lemma \ref{subduction of final diffeology},
we have a subduction
\[
(f_{1} \times 1_{Z^{\prime}}) \bigcup (f_{2} \times 1_{Z^{\prime}}) \colon (X_{1} \times Z^{\prime}) \coprod (X_{2} \times Z^{\prime})
\to (X_{1} \cup_{(i_{1}, i_{2})} X_{2}) \times Z^{\prime}.
\]
Therefore the commutative diagram above is a pushout square.
\end{proof}
%
%
%
%
%
%
\section{Smooth CW complexes}
In this section we introduce the notion of \emph{smooth CW complexes} and discuss their fundamental properties.
\begin{dfn}\label{dfn:definition of smooth CW complexes}
(i)
We say that a pair of diffeological spaces $(X,A)$ is a \emph{smooth relative CW complex} if there exists a sequence of inclusions 
called the skeleta
\[
A=X^{-1} \xrightarrow{i_{0}} X^{0} \xrightarrow{i_{1}} \cdots X^{n-1} \xrightarrow{i_{n}}  X^{n} \to \cdots
\]
such that the natural map $X^{-1} \to \mbox{colim}X^{n}$ coincides with the inclusion $A \to X$,
and for each $n \geq 0$ there exists a pushout square
\[
\xymatrix{
\coprod_{\lambda \in \Lambda_{n}}
\partial I^{n}
\ar[r]^(0.58){\bigcup \phi_{\lambda}}
\ar[d]_{\cap}
&
X^{n-1}
\ar[d]^{i_{n}}
\\
\coprod_{\lambda \in \Lambda_{n}} I^{n}
\ar[r]_(0.6){\bigcup \Phi_{\lambda}}
&
X^{n},
}
\]
that is,
the $n$-skelton $X^{n}$ is a pushouts $X^{n-1} \cup_{\bigcup \phi_{\lambda}} (\coprod_{\lambda \in \Lambda} I^{n})$.
Then,
\[
i_{n} \bigcup_{\lambda \in \Lambda_{n}} \Phi_{\lambda} \colon X^{n-1} \coprod_{\lambda \in \Lambda_{n}} I^{n} \to X^{n}
\]
is a subduction by Lemma \ref{subduction of final diffeology}.
We call 
$\phi_{\lambda}$ and $\Phi_{\lambda}$ \emph{attaching} and \emph{characteristic} maps,
respectively.
In particular,
if $A = \emptyset$ then $X$ is called a \emph{smooth CW complex}.

(ii)
We say that a smooth relative CW complex $(X,A)$ has dimension $n$,
written $\mbox{dim}(X,A)=n$,
if $X=X^{n}$ holds.

(iii)
A smooth relative CW complex is called \emph{gathered} if all attaching maps are tame.
%
\end{dfn}
We will introduce the examples of smooth CW complexes.
\begin{prp}\label{prp:example of CW complex}
The cube $I^{n}$ is a smooth CW complex and the unit interval $I$ is a gathered CW complex.
\end{prp}
\begin{proof}
For all $m$,
let $\partial I^{m} \cup_{1_{\partial I^{m}}} I^{m}$ be the adjunction space of $I^{m}$ to $\partial I^{m}$ along the identity $1_{\partial I^{m}}$.
Then natural map $\partial I^{m} \cup_{1_{\partial I^{m}}} I^{m} \to I^{m}$ is a diffeomorphism.
Thus $I^{n}$ is a smooth CW complex.
Since the identity $\partial I \to \partial I$ is tame,
$I$ is a gathered CW complex.
\end{proof}
Let $(X,A)$ be a smooth relative CW complex.
Let us denote by
\[
\{ (I^{n}, \Phi_{\lambda}) \}_{\lambda \in \Lambda_{n},n\geq -1}
\]
the functionalized cover consisting of 
all characteristic maps $\Phi_{\lambda} \colon I^{n} \to X^{n}$,
where $\Phi_{\lambda}$ is the identity $A \to A$ if $n=-1$.
We call it the characterized cover of $(X,A)$.
Then we have the following.
\begin{prp}\label{plots of smooth CW complexes}
Let $(X,A)$ be a smooth relative CW complex and let $\{ (I^{n}, \Phi_{\lambda}) \}_{\lambda \in \Lambda_{n},n\geq -1}$
be its characterized cover.
Then  the diffeology of $(X,A)$ is generated by this cover.
\end{prp}
\begin{proof}
By the properties of colimits and Lemma \ref{subduction of final diffeology},
any plot $P \colon U \to X$ of $(X,A)$ is locally,
at each $r \in U$,
along a characteristic map $\Phi_{\lambda} \colon I^{n} \to X^{n}$ as shown in the following commutative diagram
\[
\xymatrix{
V_{r}
\ar[rrrrd]^{P|V_{r}}
\ar[rrrd]
\ar[rrd]
\ar[rd]
\ar[d]
&
&
&
&
\\
I^{n}
\ar@<-0.3ex>@{^{(}->}[r]
&
\coprod I^{n}
\ar[r]_{\bigcup \Phi_{\lambda}}
&
X^{n}
\ar@<-0.3ex>@{^{(}->}[r]
&
\coprod X^{n}
\ar[r]_(0.3){\bigcup j_{n}}
&
X= \mbox{colim}_{n \geq -1}X^{n},
}
\]
where $V_{r}$ is an open neighborhood of $r$ and $j_{n}$ is the inclusion $X^{n} \to X$.
\end{proof}
Let $X$ be a smooth CW complex.
We say that a diffeological space $A$ is a \textit{subcomplex} of $X$
if it is a smooth CW complex such that all characteristic maps of $A$ are characteristic maps of $X$.
Then the pair $(X,A)$ can be viewed as a smooth relative CW complex by the following condition.
\begin{crl}
If $A$ is a subcomplex of a smooth CW complex $X$,
then it is a subspace of $X$.
\end{crl}
\begin{proof}
It is clear that $A$ is contained in $X$ and all plots of $A$ belong to the diffeology of $X$ by Proposition \ref{plots of smooth CW complexes}.
\end{proof}
\begin{prp}\label{prp:product of CW}
Let $(X,A)$ and $(Y,B)$ be smooth relative CW complexes.
Then so is $(X \times Y,A \times B)$.
\end{prp}
\begin{proof}
For $p,q \geq 0$,
let us denote 
\[
Z^{-1}=A \times B\ \mbox{and} \ Z^{n}= \cup_{n=p+q} (X^{p} \times Y^{q}),
\]
where $Z^{n}$ is the subspace of $X \times Y$.
Clearly,
$Z^{-1} \subset Z^{0}=X^{0} \times Y^{0}$ holds.
Suppose there exists the following skeleta
\[
Z^{-1} \to Z^{0} \to Z^{-1} \to \cdots \to Z^{n-1}.
\]
We shall inductively construct the $n$-skeleton $Z^{n}$ extending $Z^{n-1}$.
Let $\Phi_{\lambda} \colon I^{p} \to X^{p}$ and $\Phi^{\prime}_{\lambda^{\prime}} \colon I^{q} \to Y^{q}$ be characteristic maps of $X$ and $Y$,
respectively.
Let us denote by $\Psi_{(\lambda, \lambda^{\prime})}$ the characteristic map given by the product map
\[
\Psi_{(\lambda, \lambda^{\prime})}=\Phi_{\lambda} \times \Phi^{\prime}_{\lambda^{\prime}} \colon I^{p} \times I^{q} \cong I^{p+q} \to X^{p} \times Y^{q}.
\]
The attaching map $\psi_{(\lambda, \lambda^{\prime})}$ are given by the restriction $\Psi_{(\lambda,\lambda^{\prime})}|\partial I^{n}$,
that is,
\[
\psi_{(\lambda, \lambda^{\prime})} \colon \partial I^{n}=I^{p} \times \partial I^{q} \cup \partial I^{p} \times I^{q}
\xrightarrow{\Phi_{\lambda} \times \phi^{\prime}_{\lambda^{\prime}} \cup \phi_{\lambda} \times \Phi^{\prime}_{\lambda^{\prime}}}
X^{p} \times Y^{q-1} \cup X^{p-1} \times Y^{q},
\]
where $\phi_{\lambda} \colon \partial I^{p} \to X^{p-1}$ and $\phi^{\prime}_{\lambda^{\prime}} \colon \partial I^{q} \to Y^{q-1}$ are the attaching maps.
Since $j_{n} \bigcup_{n=p+1} \Psi_{(\lambda, \lambda^{\prime})} \colon Z^{n-1} \coprod_{n=p+q} I^{p+q} \to Z^{n}$ is a subduction by 
Proposition \ref{product of final diffeology},
Lemma \ref{subduction of final diffeology} and Proposition \ref{plots of smooth CW complexes},
the following commutative diagram 
\[
\xymatrix@C=40pt{
\coprod_{n=p+q} \partial I^{n}
\ar[r]^(0.58){\bigcup \psi_{(\lambda, \lambda^{\prime})}}
\ar[d]
&
Z^{n-1}
\ar[d]^{j_{n}}
\\
\coprod_{n=p+q} I^{n}
\ar[r]_(0.58){\bigcup \Psi_{(\lambda, \lambda^{\prime})}}
&
Z^{n}
}
\]
is a pushout square.
Then the imduced map $\mbox{colim}_{n \geq -1} Z^{n} \to X \times Y$ is a diffeomorphism by the properties of colimits.
Thus $(X \times Y, A \times B)$ is a smooth relative CW complex.
\end{proof}
\begin{rmk}
In general,
the product space of two relative gathered CW complexes is not gathered since its attaching maps are not tame.
\end{rmk}
\begin{prp}\label{prp:coproduct of CW complexes}
The coproduct of two smooth relative CW complexes is again smooth relative CW complex.
\end{prp}
\begin{proof}
Let $(X,A)$ and $(Y,B)$ be smooth relative CW complexes.
Let us denote by $Z^{n}$ the coproduct $X^{n} \coprod Y^{n}$,
where $X^{n}$ and $Y^{n}$ are $n$-skeltons.
Then it is clear that we have the following skeleta
\[
A \coprod B=Z^{-1} \to \cdots Z^{n-1} \to Z^{n} \to \cdots.
\]
Thus the coproduct $(X \coprod Y, A \coprod B)$ is a smooth relative CW complex.
\end{proof}
\begin{prp}\label{prp:quotient of CW complexes}
Let $(X,A)$ be a smooth relative CW complex.
Then the quotient space $X/A$ is a CW complex with a vertex corresponding to $A$.
\end{prp}
\begin{proof}
From the skeleta of $(X,A)$,
we naturally have the following commutative diagram
\[
\xymatrix{
A=X^{-1}
\ar[r]^{i_{0}}
\ar[d]^{\pi_{-1}}
&
X^{0}
\ar[r]^{i_{1}}
\ar[d]^{\pi_{0}}
&
\cdots
\ar[r]^(0.47){i_{n-1}}
&
X^{n-1}
\ar[r]^{i_{n}}
\ar[d]^{\pi_{n-1}}
&
X^{n}
\ar[r]^{i_{n+1}}
\ar[d]^{\pi_{n}}
&
\cdots
\\
\ast
\ar[r]_(0.45){\tilde{i}_{0}}
&
X^{0}/A
\ar[r]_(0.52){\tilde{i}_{1}}
&
\cdots
\ar[r]_(0.4){\tilde{i}_{n-1}}
&
X^{n-1}/A
\ar[r]_(0.52){\tilde{i}_{n}}
&
X^{n}/A
\ar[r]_(0.54){\tilde{i}_{n+1}}
&
\cdots,
}
\]
where $\pi_{n} \colon X^{n} \to X^{n}/A$ is the subduction.
By the properties of subductions,
for each $n$,
the following commutative diagram
\[
\xymatrix@C=60pt{
\coprod_{\lambda \in \Lambda_{n}}
\partial I^{n}
\ar[d]_{\cap}
\ar[r]^{\pi_{n-1} \circ (\cup \phi_{\lambda})}
&
X^{n-1}/A
\ar[d]^{\tilde{i}_{n}}
\\
\coprod_{\lambda \in \Lambda_{n}} I^{n}
\ar[r]_{\pi_{n} \circ (\cup \Phi_{\lambda})}
&
X^{n}/A
}
\]
is a pushout square.
Then the induced map $\mbox{colim}X^{n}/A \to X /A$ is a diffeomorphism because it is a bijective subduction.
Thus $X/A$ is a smooth CW complex.
\end{proof}
Next,
we study the $D$-topology of smooth CW complexes.
Any diffeological space $X$ determines a
topological space $TX$ having the same underlying set as $X$ and is
equipped with the final topology with respect to the plots of $X$.
This topology is called the $D$-topology of $X$ and its open sets are called $D$-open sets.
Any smooth map $f \colon X \to Y$ induces a continuous map
$TX \to TY$, hence we have a functor $T \colon \Diff \to \Top$ (cf.~\cite{Zem85}).
On the other hand,
suppose $X$ is a topological space.
Let us denote by $DX$ the diffeological space with
the same underlying set as $X$ and with the diffeology consisting of
all continuous maps from an open subset of a Euclidean space into $X$.
Clearly, every continuous map $X \to Y$ induces a smooth map
$DX \to DY$, hence there is a functor $D \colon \Top \to \Diff$ which
takes a topological space $X$ to $DX$ (cf.~\cite{BH}).
Then the functor $T$ is a left adjoint to $D$ (cf.~{\cite[Proposition 2.1]{SYH}}).
Thus $T$ preserves colimits and $D$ preserves limits.

A topological space is called a paracompactum if it is a paracompact Hausdorff space.
We know that a topological CW complex is a paracompactum (cf.~\cite{Miyazaki}).
We can use the arguments described in \cite{HS1} 
since any smooth relative CW complex is a smooth relative cell complex (cf.~{\cite[Definition 1]{HS1}}).
Then we have the followings 
\begin{prp}[{\cite[Proposition 2 and Proposition 8]{HS1}}]\label{topological CW complexes}
Let $(X,A)$ be a smooth relative CW complex.
Then $(TX,TA)$ is a topological relative CW complex.
If $TA$ is a paracompactum,
then so is $(TX,TA)$.
\end{prp}
\begin{proof}
Clearly,
$(TX,TA)$ is a topological relative CW complex.
For each $n \geq 0$,
the pair $(TI^{n}, T\partial I^{n})$ is homeomorphic to 
the pair $(I^{n}, \partial I^{n})$ of topological subspaces of $\mathbf{R}^{n_{\beta+1}}$
by {\cite[Lemma 3.16]{DGE}}.
Since the functor $T$ preserves the colimits,
we have the following pushout square
\[
\xymatrix{
\coprod_{\lambda \in \Lambda_{n}}
\partial I^{n}
\ar[r]^(0.58){\bigcup \phi_{\lambda}}
\ar[d]_{\cap}
&
TX^{n-1}
\ar[d]^{i_{n}}
\\
\coprod_{\lambda \in \Lambda_{n}} I^{n}
\ar[r]_(0.6){\bigcup \Phi_{\lambda}}
&
TX^{n}.
}
\]
In general,
the coproduct in $\bf Top$ of small family of paracompacta is a paracompactum.
Hereinafter,
we can prove that $(TX,TA)$ is a paracompactum as well as \cite[Proposition 8]{HS1}.
\end{proof}
\begin{prp}\label{the weak topology induced smooth CW complexes}
Let $X$ be a smooth CW complex and $\{ ( I^{n}, \Phi_{\lambda}) \}_{\lambda \in \Lambda_{n}, n \geq 0}$ be its characterized cover.
Then a topological CW complex $TX$ has the weak topology with respect to its covering $\{T \Phi_{\lambda}(I^{n}) \}$.
\end{prp}
\begin{proof}
By Proposition \ref{topological CW complexes},
$T \Phi_{\lambda}(I^{n})$ is a subspace of $TX$ for each $\lambda$.
Let $A$ be a subspace of $TX$ such that $T\Phi_{\lambda}(I^{n}) \cap A$ is a $D$-open set of $T\Phi_{\lambda}(I^{n})$ for each $\lambda$.
Let $P \colon U \to X$ be a plot of $X$.
For any $r \in P^{-1}(A)$,
there are a plot $Q_{r} \colon V_{r} \to I^{n}$ and a characteristic map $\Phi_{\lambda} \colon I^{n} \to X^{n}$ satisfying 
$P|V_{r}=\Phi_{\lambda} \circ Q_{r}$ by Proposition \ref{plots of smooth CW complexes}.
Then we have
\[
r \in (\Phi_{\lambda} \circ Q_{r})^{-1}(A) \subset P^{-1}(A).
\]
Thus $A$ is a $D$-open set of $TX$.
\end{proof}
\begin{thm}[{\cite[Theorem 9]{HS1}}]
Let $X$ be a smooth CW complex.
Then every its $D$-open covering $\{ U_{\lambda} \}_{\lambda \in \Lambda}$ has a subordinate partition of unity 
$\{\psi_{\lambda} \colon U_{\lambda} \to \mathbf{R} \}_{\lambda \in \Lambda}$.
\end{thm}
\begin{thm}[{\cite[Theorem 5]{HS1}}]\label{thm:Whitney approximation property}
Let $(X,A)$ and $(Y,B)$ be relative gathered CW complexes.
Let $f \colon (TX,TA) \to (TY,TB)$ be a continuous map such that its restriction to $A$ is a smooth map.
Then there exists a smooth map $g \colon (X,A) \to (Y,B)$ such that $f$ and $Tg$ are homotopic relative to $TA$.
\end{thm}
\begin{prp}[{\cite[Proposition 6.8]{HS}}]
Let $X$ be a gathered CW complex.
Then the unit $X \to DTX$ is a weak homotopy equivalence (Definition \ref{dfn:weak homotopy equivalence}).
\end{prp}
We introduce one result of Theorem A.1 in {\cite[Appendix A]{Iwase}}.
\begin{thm}[{\cite[Theorem 6]{HS1}}]
Let $X$ be a topological CW complex.
Then there exists a gathered CW complex $Y$ such that $X$ is continuously homotopy equivalent to $TY$.
\end{thm}
\begin{rmk}
In \cite{Kihara1} and \cite{Kihara},
Kihara gives the notion of smooth CW complexes with characteristic maps of simplexes in the process of introducing a model structure on the category of diffeological spaces.
In {\cite[Appendix A]{Iwase}},
Iwase and Izumida introduce smooth CW complexes with characteristic maps of disks.
Since simplexes, disks and cubes are not diffeomorphic,
the characteristic maps need to be properly defined for the purpose of discussion.
\end{rmk}
%
%
%
%
%
%
%
%
\section{The homotopy extension property}
In this section we prove that relative gathered CW complexes have the homotopy extension property.
Moreover we develop this theory by giving the notion of $N$-connectedness for diffeological spaces.
First,
we will prepare several lemmas.
Because we must be careful to discuss this theory since all characteristic maps of gathered CW complexes are not tame.
\begin{lmm}\label{lmm:smooth}
Let $h_{0} \colon X \to Y$ be a smooth map between diffeological spaces.
Let $A$ be a subspace of $X$ and let $H \colon A \times I \to Y$ be an $\epsilon$-tame homotopy satisfying $H(a,0)=h_{0}(a)$ for all $a \in A$.
Then a map
$
\tilde{H} \colon (A \times I) \cup (X \times \{0\}) \to Y
$
given by the formula
\[
\tilde{H}(x,t)=
\left\{
\begin{array}{lll}
H(x,t),
&
(x,t) \in A \times I
\\
h_{0}(x),
&
(x,t) \in X \times \{0 \}
\end{array}
\right.
\]
is smooth,
where $ (A \times I) \cup (X \times \{0\}) $ is subspace of $X \times I$.
\end{lmm}
Hereafter,
all homotopies are defined by tame,
so we will use this property without referring to the lemma above.
\begin{proof}
Let $P \colon U \to A \times I \cup (X \times \{0\}$ be a plot given by $P(r)=(P_{1}(r),P_{2}(r))$,
that is,
$P=(P_{1} \times P_{2}) \circ \Delta$ holds,
where $\Delta \colon U \to U \times U$ is the diagonal.
Then we have
\[
\tilde{H} \circ P(r)=
\left\{
\begin{array}{lll}
H(P_{1}(r),P_{2}(r)),
&
P_{2}(r) \in (0,1]
\\
h_{0}(P_{1}(r)),
&
P_{2}(r) \in [0, \epsilon),
\end{array}
\right.
\]
for all $r \in U$.
Let $V_{1}$ and $V_{2}$ be open sets $P_{2}^{-1}((0,1])$ and $P_{2}^{-1}([0,\epsilon))$ of $U$,
respectively.
Then the restrictions $\tilde{H} \circ P|V_{1} = H \circ (P_{1} \times P_{2}) \circ \Delta|V_{1}$ and 
$\tilde{H} \circ P|V_{2}=h_{0} \circ P_{1}|V_{2}$ are plots of $Y$.
Thus $\tilde{H}$ is smooth.
\end{proof}
\begin{lmm}\label{prp:smooth retraction}
Let $f \colon L^{n-1} \to X$ be a smooth map such that the restriction of $f$ to $\partial I^{n-1}\times I$ is $\epsilon$-tame.
Then $f$ can be extended to a tame homotopy $g\colon I^{n-1} \times I \to X$.
\end{lmm}
\begin{proof}
Let us define a tame homotopy $F \colon I^{n-1} \times [0,{\epsilon}/{2}] \to X$ by the formula
\[
F(t,s)=
f\left(\tilde{T}(s)t+(1-\tilde{T}(s))T^{n-1}_{{\epsilon}/{2},\epsilon}(t)\right),
\]
where $\tilde{T} \colon [0, {\epsilon}/{2}] \to I$ is a tame map given by $\tilde{T}(s)=T_{\epsilon,{1}/{2}}({2s}/{\epsilon} )$.
Then it satisfies $F(t,s)=f(t)$ for $(t,s) \in \partial I^{n-1} \times [0,{\epsilon}/{2}]$.
Let
\[
F^{\prime} \colon \partial I^{n-1} \times [{\epsilon}/{2},1] \cup I^{n-1} \times \{{\epsilon}/{2}\} \to X
\]
be a smooth map given by the formula
\[
F^{\prime}(t,s)=
\left\{
\begin{array}{lll}
f(t,s), & (t,s) \in \partial I^{n-1} \times [{\epsilon}/{2},1] \\
f \circ T^{n-1}_{{\epsilon}/{2},\epsilon}(t), & (t,s) \in I^{n-1} \times \{{\epsilon}/{2}\}.
\end{array}
\right.
\]
Then there exists a tame map $\tilde{F} \colon I^{n-1} \times [{\epsilon}/{2},1] \to X$ extending $F^{\prime}$ 
by Proposition \ref{extend of J}.
Thus we have a tame homotopy $g \colon I^{n-1} \times I \to X$ given by the formula
\[
g(t,s)=
\left\{
\begin{array}{ll}
F(t,s), & 0 \leq s \leq {\epsilon}/{2} \\
\tilde{F}(t,s),& {\epsilon}/{2} \leq s \leq 1,
\end{array}
\right.
\]
which extends $f$.
\end{proof}
Then we have the following homotopy extension property.
\begin{lmm}[Homotopy extension property]\label{lmm:HEP}
Let $(X,A)$ be a relative gathered CW complex.
Suppose we are given a smooth map $f \colon X \to Y$ and a tame homotopy $h \colon A \times I \to Y$ satisfying $h(a,0)=f(a)$.
Then there exists a tame homotopy $H \colon X \times I \to Y$ satisfying $H|A \times I=h$ and $H(x,0)=f(x)$.
\end{lmm}
\begin{proof}
Inductively,
suppose there is a tame homotopy $H^{n-1} \colon X^{n-1} \times I \to Y$ satisfying 
$H^{n-1}(x,0)=f(x)$ and $H^{n-1}|X^{n-2} \times I=H^{n-2}$.
For each $\lambda \in \Lambda_{n}$,
let $\phi_{\lambda} \colon \partial I^{n} \to X^{n-1}$ and $\Phi_{\lambda} \colon I^{n} \to X^{n}$ be attaching and characteristic maps,
respectively.
We have a smooth map $H_{\lambda} \colon \partial I^{n} \times I \cup I^{n} \times \{0\} \to Y$ given by the formula
\[
\tilde{H}_{\lambda}(t,s)=
\left\{
\begin{array}{llll}
H^{n-1}(\phi_{\lambda}(t),s), & (t,s) \in \partial I^{n} 
\\
f(\Phi_{\lambda}(t)), & (t,s) \in I^{n-1} \times \{0\}.
\end{array}
\right.
\]
Then there exists a tame homotopy $\tilde{H}_{\lambda} \colon I^{n} \times I \to Y$ extending $H_{\lambda}$ 
by Lemma \ref{prp:smooth retraction}.
Hence we have a tame homotopy $H^{n} \colon X^{n} \times I \to Y$ extending $H^{n-1}$ such that the following diagram is commutative
\[
\xymatrix@C=60pt{
(X^{n-1} \times I) \coprod_{\lambda \in \Lambda_{n}} (I^{n} \times I)
\ar[r]^(0.72){H^{n-1} \bigcup_{\lambda}\tilde{H}_{\lambda}}
\ar[d]_{(i_{n} \times 1_{I}) \bigcup_{\lambda}(\Phi_{\lambda} \times 1_{I})}
&
Y
\\
X^{n} \times I,
\ar@{.>}[ru]_{\exists H^{n}}
}
\]
since $(i_{n} \times 1_{I}) \bigcup_{\lambda}(\Phi_{\lambda} \times 1_{I})$ is a subduction by Lemma \ref{subduction of final diffeology}.
Similarly,
there exists a tame homotopy $H \colon X \times I \to Y$ making the diagram commutative
\[
\xymatrix@C=50pt{
\coprod_{n \geq -1} (X^{n} \times I)
\ar[r]^(0.68){\bigcup_{n} H^{n}}
\ar[d]_{\bigcup_{n}( j_{n} \times 1_{I})}
&
Y
\\
X \times I,
\ar@{.>}[ru]_{\exists H}
}
\]
since $ \bigcup_{n}(j_{n} \times 1_{I}) \colon \coprod_{n \geq -1}(X^{n} \times I) \to X \times I$ is a subducition by Lemma \ref{subduction of final diffeology},
where $j_{n} \colon X^{n} \to X$ is the inclusion for each $n \geq -1$.
That is,
$H$ is given by $H|X^{n} \times I=H^{n}$.
Then $H(x,0)=f(x)$ holds.
\end{proof}
\begin{thm}\label{HEP2}
Let $(X,A)$ be a gathered relative CW complex and $(Y,B)$ be a pair of diffeological spaces.
Suppose there are a smooth map between pairs $f \colon (X,A) \to (Y,B)$ and a tame homotopy of pairs $F \colon X \times I \to Y$ satisfying
$F(x,0)=f(x)$ and $F(x,1) \in B$ for all $x \in X$.
Then there is a tame homotopy $F^{\prime} \colon X \times I \to Y$ relative to $A$ satisfying 
$F^{\prime}(x,0)=f(x)$ and $F^{\prime}(x,1) \in B$ for all $x \in X$.
\end{thm}
\begin{proof}
Let $G \colon A \times I \to B$ be a tame homotopy given by $G(a,t)=F(a,1-t)$.
Then there is a tame homotopy of pairs $H \colon X \times I \to B$ satisfyng $H|A \times I=G$ and $H(x,0) \in B$ by Lemma \ref{lmm:HEP}.
Let us define $H^{\prime} \colon X \times I \to Y$ by the formula
\[
H^{\prime}(x,t)=\left\{
\begin{array}{lll}
F(x,2t),
&
0 \leq t \leq 1/2
\\
H(x,2t-1),
&
1/2 \leq t \leq 1.
\end{array}
\right.
\]
Let $\mu(t)=T_{\sigma, \tau}(2t)-T_{\sigma, \tau}(2t-1)$ and define $G^{\prime} \colon A \times I \times I \to B$ by
$G^{\prime}(a,t,s)=H^{\prime}(a,(1-T_{\sigma, \tau}(s)) \mu(t))$.
Then $G^{\prime}$ gives
$
G^{\prime}(a,t,0)=H^{\prime}(a,\mu(t)),\
G^{\prime}(a,t,1)=f(a)
$
and $G^{\prime}(a,t,s)=f(a)$ for any $(a,t,s) \in A \times \partial I \times I.$
Then there is a tame homotopy $\tilde{H} \colon X \times I \times I \to Y$ satisfying
$\tilde{H}(x,t,0)=H^{\prime}(x, \mu(t))$ and $\tilde{H}|A \times I \times I=G^{\prime}$ by Lemma \ref{lmm:HEP}.
Thus we have a smooth map $F^{\prime} \colon X \times I \to Y$ given by the formula
\[
F^{\prime}(x,t)=
\left\{
\begin{array}{lll}
\tilde{H}(x,0,3t), & 0 \leq t \leq 1/3 \\
\tilde{H}(x,3t-1,1), & 1/3 \leq t \leq 2/3 \\
\tilde{H}(x,1,3-3t), & 2/3 \leq t \leq 1.
\end{array}
\right.
\]
Then $F^{\prime}$ is a tame homotopy relative to $A$ satisfying $F^{\prime}(x,0)=f(x)$ and $F^{\prime}(x,1) \in B$.
\end{proof}
In {\cite[Chapter 10]{May}},
the following connectedness is introduced.
\begin{dfn}\label{dfn:N-connected}
We say that a diffeological space $X$ is $N$-connected if $\pi_{n}(X,x_{0})=0$ for each $0\leq n \leq N$.
Clearly,
$X$ is $0$-connected if and only if it is path connected.
We say that a pair of diffeological spaces $(Y,B)$ is $N$-connected if $\pi_{0}B \to \pi_{0}Y$ is surjective and
$\pi_{n}(Y,B,b)=0$ for $1 \leq n \leq N$ and all $b \in B$.
\end{dfn}
We have the basic extension property.
\begin{prp}
Let $(X,A)$ be a $N$-dimensional relative gathered CW complex and let $Y$ be a $(N-1)$-connected space.
Then any smooth map $f \colon A \to Y$ can be extended to a smooth map $g \colon X \to Y$.
\end{prp}
\begin{proof}
Suppose we inductively construct a smooth map $g^{n-1} \colon X^{n-1} \to Y$ satisfying $g^{n-1} |X^{n-2}=g^{n-2}$.
For each $\lambda \in \Lambda_{n}$,
let $\phi_{\lambda} \colon \partial I^{n} \to X^{n-1}$ and $\Phi_{\lambda} \colon I^{n} \to X^{n}$ be attaching and characteristic maps,
respectively.
Since $Y$ is $(N-1)$-connected,
there exists a tame map $H_{\lambda} \colon \partial I^{n} \times I \to Y$ between $g^{n-1} \circ \phi_{\lambda}$ and the constant map $C_{y_{0}}$ with value $y_{0} \in Y$.
We define a tame map $H^{\prime}_{\lambda} \colon J^{n} \to Y$ by putting
\[
H^{\prime}_{\lambda}(t,s)=
\left\{
\begin{array}{lll}
H_{\lambda}(t,s), & (t,s) \in \partial I^{n} \times I \\
y_{0}, & (t,s) \in I^{n} \times \{1\}.
\end{array}
\right.
\]
Then there is a tame map $\tilde{H}_{\lambda} \colon I^{n} \times I \to Y$ extending $H^{\prime}_{\lambda}$ by Lemma \ref{extend of J}.
Let $g_{\lambda} \colon I^{n} \to Y$ be a smooth map given by $g_{\lambda}(t)=\tilde{H}_{\lambda}(t,1)$ for all $t \in I^{n}$.
By Lemma \ref{subduction of final diffeology},
there is a smooth map $g^{n} \colon X^{n} \to Y$ such that the following diagram commutes
\[
\xymatrix@C=60pt{%
X^{n-1} \coprod_{\lambda \in \Lambda_{n}} I^{n} 
\ar[r]^-{g^{n-1} \bigcup_{\lambda} g_{\lambda}} 
\ar[d]_-{i_{n} \bigcup_{\lambda} \Phi_{\lambda}} & Y
\\
X^{n}.
\ar@{.>}[ru]_-{\exists g^{n}}
}
\]
Then we have a smooth map $g \colon X \to Y$ given by $g|X^{n}=g^{n}$.
\end{proof}
\begin{prp}\label{prp:homotopic}
Let $(X,A)$ be a $N$-dimensional relative gathered CW complex and let $f \colon (X,A) \to (Z,C)$ be a smooth map of pairs.
If a pair $(Z,C)$ is $N$-connected,
there exists a smooth map of pairs $g \colon (X,A) \to (Z,C)$ satisfying $g(X) \subset C$ and a tame homotopy of pairs $f \simeq g$ relative to $A$.
\end{prp}
\begin{proof}
Since a pair $(Z,C)$ is $N$-connected,
for any $x \in X^{0} \setminus A$
there exists a tame path $f_{x} \colon I \to Z$ satisfying $f_{x}(0)=f(x)$ and $f_{x}(1) \in C$.
Let $F^{0} \colon X^{0} \times I \to Z$ be a tame homotopy relative to $A$ given by the formula
\[
F^{0}(x,t)=
\left\{
\begin{array}{lll}
f_{x}(t), & (x,t) \in (X^{0} \setminus A) \times I \\
f(x), & (x,t) \in A \times I,
\end{array}
\right.
\]
then it is smooth since $X^{0} \setminus A$ is discrete.
Suppose we inductively construct a smooth map of pairs $g^{n-1} \colon (X^{n-1},A) \to (Z,C)$ satisfying $g^{n-1}(X^{n-1}) \subset C$
and a tame homotopy of pairs $F^{n-1} \colon f|X^{n-1} \simeq g^{n-1}$ relative to $A$ satisfying $F^{n-1}|X^{n-2} \times I =F^{n-2}$.
Let $\phi_{\lambda} \colon \partial I^{n} \to X^{n-1}$ and $\Phi_{\lambda} \colon I^{n} \to X^{n}$ be attaching and characteristic maps,
respectively.
Let $F_{\lambda} \colon L^{n} \to Z$ be a smooth map given by the formula
\[
F_{\lambda}(t,s)=
\left\{
\begin{array}{lll}
F^{n-1}(\phi_{\lambda}(t),s), & (t,s) \in \partial I^{n} \times I \\
f(\Phi_{\lambda}(t)), & (t,s) \in I^{n} \times \{0\}.
\end{array}
\right.
\]
Then $F_{\lambda}$ can be extended to be an $\epsilon$-tame homotopy $\tilde{F}_{\lambda} \colon I^{n} \times I \to Z$.
Let $f_{\lambda} \colon (I^{n}, \partial I^{n}) \to (Z,C)$ be a smooth map of pairs given by 
$f_{\lambda}(t)=\tilde{F}_{\lambda}(t,1)$ for any $t \in I$.
Since $(Z,C)$ is $N$-connected,
there exists a tame homotopy of pairs $\tilde{F}^{\prime}_{\lambda} \colon I^{n} \times I \to Z$ satisfying 
$\tilde{F}^{\prime}_{\lambda}(t,0)=f_{\lambda}(t)$ and $\tilde{F}^{\prime}_{\lambda}(t,1) \in C$.
By Theorem \ref{HEP2},
there are smooth map of pairs $g_{\lambda} \colon (I^{n},\partial I^{n}) \to (Z,C)$ satisfying $g_{\lambda}(I^{n}) \subset C$ and a tame homotopy 
of pairs $\tilde{F}^{\prime \prime} \colon I^{n} \times I \to Z$ relative to $\partial I^{n}$ from $f_{\lambda}$ to $g_{\lambda}$.
Then we have a tame homotopy $G_{\lambda} \colon f \circ \Phi_{\lambda} \simeq g_{\lambda}$ relative to $\partial I^{n}$ given by the formula
\[
G_{\lambda}(t,s)=
\left\{
\begin{array}{lll}
\tilde{F}_{\lambda}(t,s), & 0 \leq s \leq ({2-\epsilon})/{2} \\
\tilde{F}^{\prime \prime}_{\lambda}(t, ({2s-2+\epsilon})/{\epsilon} ), & ({2-\epsilon})/{2} \leq s \leq 1.
\end{array}
\right.
\]
Then $G_{\lambda}|L^{n}=F_{\lambda}$ holds.
By Lemma \ref{subduction of final diffeology},
we have a tame homotopy of pairs $F^{n} \colon f|X^{n} \simeq g^{n}$ relative to $A$ such that the following diagram commutes
\[
\xymatrix@C=60pt{%
(X^{n-1} \times I)\coprod_{\lambda \in \Lambda_{n}} (I^{n} \times I)
\ar[r]^-{F^{n-1} \bigcup_{\lambda} G_{\lambda}} 
\ar[d]_-{(i_{n} \times 1_{I}) \bigcup_{\lambda}( \Phi_{\lambda} \times 1_{I}) } & Z
\\
X^{n} \times I
\ar@{.>}[ru]_-{\exists {F}^{n}}, }%
\]
where $g^{n} \colon (X^{n},A) \to (Z,C)$ is a smooth map of pairs given by $g^{n}(x)=F^{n}(x,1)$ for any $x \in X^{n}$.
Hence we have a smooth map of pairs $g \colon (X,A) \to (Y,B)$ satisfying $g(X) \subset C$ and 
a tame homotopy of pairs $F \colon f \simeq g$ relative to $A$ given by $g|X^{n}=g^{n}$ and $F|X^{n} \times I=F^{n}$,
respectively.
\end{proof}
%
%
%
%
%
%
%
%
\section{The homotopy extension lifting property and the Whitehead theorem}
In this section we will shall show that gathered CW complexes have the homotopy extension property and satisfy the Whitehead theorem.
First,
we give the following notion. 
\begin{dfn}\label{dfn:weak homotopy equivalence}
We say that a smooth map $\phi \colon X \to Y$ between diffeological spaces is a $N$-equivalence if for all $x \in X$ the induced map
\[
\phi_{\ast} \colon \pi_{n}(X,x) \to \pi_{n}(Y, \phi(x))
\]
is injective for $0 \leq n < N$ and surjective for $n<N$.
A smooth map $\phi \colon (X,A) \to (Y,B)$ between pairs of diffeological spaces is called a $N$-equivalence if $\phi \colon X \to Y$ and $\phi_{0} \colon A \to B$ are $N$-equivalences,
where $\phi_{0}$ is the restriction $\phi|A$.
If they are $N$-equivalences for all $n$,
$\phi$ is said to be a weak homotopy equivalence.
\end{dfn}
Let $\phi \colon X \to Y$ be a smooth map and let $M_{\phi}$ be the mapping cylinder given by the following pushout square
\[
\xymatrix{
X
\ar[r]^{\phi}
\ar[d]_{j_{0}}
&
Y
\ar[d]^{i}
\\
X \times I
\ar[r]_{\Phi}
&
M_{\phi},
}
\]
where $j_{0}(x)=(x,0)$ for any $x \in X$.
Then we put the subduction $\pi_{\phi}=i \bigcup \Phi \colon Y \coprod (X \times I) \to M_{\phi}$.
Let us define $j_{1} \colon X \to X \times I$ and $\gamma^{\prime} \colon X \times I \to Y$ by 
$j_{1}(x)=(x,1)$ and $\gamma^{\prime}(x,t)=\phi(x)$,
respectively.
We have the inclusion $j \colon X \to M_{\phi}$ and a retraction $\gamma \colon M_{\phi} \to Y$ induced by 
$j_{1}$ and $1_{Y} \bigcup \gamma^{\prime} \colon Y \coprod (X \times I) \to Y$,
respectively.
Then $\phi=\gamma \circ j$ holds.
Clearly,
$Y$ is a deformation retract of $M_{\phi}$.
\begin{lmm}\label{lmm:N-equivalence}
Let $(Z,C)$ be a relative gathered CW complex of dimension $\leq N$.
Let $\phi \colon X \to Y$ be a $N$-equivalence.
Let $f^{\prime} \colon C \to X$ and $g \colon Z \to Y$ be two smooth maps such that $\phi \circ f^{\prime}=g|C$ holds.
Then there exists a smooth map $f \colon Z \to X$ such that $\phi \circ f$ and $g$ are homotopic relative to $C$.
\end{lmm}
\begin{proof}
Let $M_{\phi}$ be the mapping cylinder of $\phi$.
Since $\phi$ is a $N$-equivalence,
for every $x \in X$ the induced map
\[
j_{\ast} \colon \pi_{n}(X,x) \to \pi_{n}(M_{\phi},x)
\]
is bijective for $0 \leq n <N$ and surjective for $n=N$. 
Then we have $\pi_{n}(M_{\phi},X,x)=0$ for $1 \leq n \leq N$ by the properties of homotopy exact sequence (cf.~{\cite[Proposition 3.14]{HS}}).
Moreover the induced map $j_{\ast} \colon \pi_{0}(X) \to \pi_{0}(M_{\phi})$ is surjective since $\phi$ is $N$-equivalence.
Thus a pair $(M_{\phi},X)$ is $N$-connected.
Le us define a tame homotopy $H \colon C \times I \to M_{\phi}$ by the composition $\pi_{\phi} \circ (f^{\prime} \times T_{\epsilon,1/2})$.
Then for any $z \in C$,
we have
\[
H(z,0)=\pi_{\phi}(f^{\prime},0)=\phi f^{\prime}(z)=g(z) \ \mbox{and} \ H(z,1)= \pi_{\phi}(f^{\prime}(z),1)=f^{\prime}(z).
\]
By Lemma \ref{lmm:HEP},
there exists a tame homotopy $H^{\prime} \colon Z \times I \to M_{\phi}$ satisfying $H^{\prime}(z,0)=j g(z)$ and $H^{\prime}|C \times I=H$.
Let us define a smooth maps of pairs $g^{\prime} \colon (Z,C) \to (M_{\phi},X)$ by putting $g^{\prime}(z)=H^{\prime}(z,1)$ for $z \in Z$.
By Proposition \ref{prp:homotopic} there exists a smooth map of pairs $f \colon (Z,C) \to (M_{\phi},X)$ satisfying
$f(Z) \subset X$ and $g^{\prime} \simeq f$ rel $C$.
Then we have
\[
f|C=g^{\prime}|C=f^{\prime} \ \mbox{and} \ g \simeq \gamma \circ g^{\prime} \simeq \gamma \circ f =\phi \circ f \ \mbox{rel}\ C.
\]
\end{proof}
Then we have the following.
\begin{thm}[Homotopy extension lifting property]\label{HELP}
Let $(X,A)$ be a relative CW complex of dimension $\leq N$ and let $p \colon Y \to Z$ be a $N$-equivalence.
Then given smooth maps $f^{\prime} \colon A \to Y,\ g \colon X \to Z$ and a tame homotopy $F^{\prime} \colon A \times I \to Z$ satisfying 
$g|A=F^{\prime} \circ i_{0}$ and $p \circ f^{\prime}=F^{\prime} \circ i_{1}$ in the following diagram,
there are a smooth map $f \colon X \to Y$ and a tame homotopy 
$G \colon X \times I \to Z$ such that the following diagram commutes
\[
\xymatrix{
A
\ar[rr]^(0.45){i_{0}}
\ar[dd]_{j}
&
&
A \times I
\ar[dd]^(0.7){j \times 1}
\ar[ld]_{F^{\prime}}
&
&
A
\ar[ll]_(0.4){i_{1}}
\ar[dd]^{j}
\ar[ld]_{f^{\prime}}
\\
&
Z
&
&
Y
\ar[ll]_(0.3){p}
&
\\
X
\ar[ru]^(0.55){g}
\ar[rr]_(0.45){i_{0}}
&
&
X \times I
\ar@{.>}[lu]^{\exists G}
&
&
X.
\ar[ll]^{i_{1}}
\ar@{.>}[lu]^{\exists f}
}
\]
\end{thm}
\begin{proof}
Inductively,
suppose there are a smooth map $f^{n-1} \colon X^{n-1} \to Y$ and a tame homotopy $G^{n-1} \colon X^{n-1} \times I \to Z$ satisfying 
$G^{n-1}(x,0)=g(x)$,
$G^{n-1}(x,1)=pf^{n-1}(x)$,
$G^{n-1}|X^{n-2} \times I=G^{n-2}$ and $f^{n-1}|X^{n-2}=f^{n-2}$.
For any $\lambda \in \Lambda_{n}$,
let $\phi_{\lambda} \colon \partial I^{ n} \to X^{n-1}$ and $\Phi_{\lambda} \colon I^{n} \to X^{n}$ be attaching and characteristic map,
respectively.
Let $K_{\lambda} \colon L^{n} \to Z$ be a smooth map given by the formula 
\[ 
K_{\lambda}(t,s)=\left\{
\begin{array}{llll}
G^{n-1}(\phi_{\lambda}(t),s),
&
(t,s) \in \partial I^{n} \times I
\\
g\Phi_{\lambda}(t),
&
(t,s) \in I^{n} \times \{0 \}
\end{array}
\right.
\]
Then there exists a tame homotopy $K^{\prime}_{\lambda} \colon I^{n} \times I \to Z$ 
extending $K_{\lambda}$ by Lemma \ref{prp:smooth retraction}.
Let $g^{\prime}_{\lambda} \colon I^{n} \to Z$ be a tame map given by $g^{\prime}_{\lambda}(t)=K^{\prime}_{\lambda}(t,1)$ for any $t \in I^{n}$.
Then we have the following commutative diagram
\[
\xymatrix@C=50pt{
\partial I^{n}
\ar[r]^{f^{n-1}\circ \phi_{\lambda}}
\ar[d]_{i_{1}}
&
Y
\ar[d]^{p}
\\
I^{n}
\ar[r]_{g^{\prime}_{\lambda}}
&
Z.
}
\]
By Lemma \ref{lmm:N-equivalence},
there are smooth map $f_{\lambda} \colon I^{n} \to Y$ and an $\epsilon$-tame homotopy 
$K^{\prime \prime}_{\lambda} \colon g^{\prime}_{\lambda} \simeq p \circ f_{\lambda}$ relative to $\partial I^{n}$.
Then we have a tame homotopy $G_{\lambda} \colon g \circ \Phi_{\lambda} \simeq p \circ f_{\lambda}$ extending $K_{\lambda}$ 
given by the formula
\[
G_{\lambda}(t,s)=\left\{
\begin{array}{lll}
K^{\prime}_{\lambda}(t,s),
&
0 \leq s \leq (2-\epsilon)/2
\\
K^{\prime \prime}_{\lambda}(t, (2s-2+\epsilon)/\epsilon),
&
(2-\epsilon)/2\leq s \leq 1.
\end{array}
\right.
\]
Since 
the following smooth maps
\begin{enumerate}
\item
$i_{n} \bigcup_{\lambda} \Phi_{\lambda} \colon X^{n-1} \coprod_{\lambda \in \Lambda_{n}} I^{n} \to X^{n}$ and
\item
$
(i_{n} \times 1_{I}) \bigcup_{\lambda} (\Phi_{\lambda} \times 1_{I})
\colon (X^{n-1} \times I) \coprod_{\lambda \in \Lambda_{n}} (I^{n} \times I) \to X^{n} \times I
$
\end{enumerate}
are subductive by Lemma \ref{subduction of final diffeology},
we have a smooth map $f^{n} \colon X^{n} \to Y$ and a tame homotopy $G^{n} \colon X^{n} \times I \to Z$ induced by 
$f^{n-1} \bigcup_{\lambda} f_{\lambda}$ and $G^{n-1} \bigcup_{\lambda} G_{\lambda}$,
respectively.
Clearly,
$G^{n}(x,0)=g(x)$,
$G^{n}(x,1)=pf^{n}(x)$,
$G^{n}|X^{n-1} \times I=G^{n-1}$ and $f^{n}|X^{n-1}=f^{n-1}$ hold.
Hence we have a smooth map $f \colon X \to Y$ and a tame homotopy $G \colon X \times I \to Z$ such that the diagram above commutes.
\end{proof} 
Next,
we will introduce the following direct and important applications of the homotopy extension lifting property.
The cube $I^{n}$ is subdivided into the subcubes
\[
K_{J_{k}}=
\left[
\frac{j_{1}-1}{k}, \frac{j_{1}}{k}
\right]
\times \cdots \times
\left[
\frac{j_{n}-1}{k}, \frac{j_{n}}{k}
\right],
\]
where $J_{k}=(j_{1}, \cdots , \ j_{n}) \in \{1, \cdots, \ k \}^{n}$.
To prove Theorem \ref{thm:extends weak} and Lemma \ref{lmm:cellular of cube},
we will give the following smooth map $\mathbb{T}^{n}_{k,\sigma, \tau} \colon I^{n} \to I^{n}$ such that its restriction to $K_{J_{k}}$ is tame for
each $J_{k}$.
\begin{lmm}\label{lmm:locally tame map}
Suppose $0 \leq \sigma < \tau \leq 1/2$.
Then there exists a locally tame map $\mathbb{T}^{n}_{k, \sigma, \tau} \colon I^{n} \to I^{n}$ satisfying the following conditions.
\begin{enumerate}
\item
For all $J_{k}$,
the restrctions $\mathbb{T}^{n}_{k, \sigma, \tau}|K_{J_{k}} \colon K_{J_{k}} \to K_{J_{k}}$ is surjective and tame under the evident diffeomorphism
$I^{n} \cong K_{J_{k}}$.
\item
The identity $1_{I^{n}}$ and $\mathbb{T}^{n}_{k,\sigma, \tau}$ are homotopic.
\end{enumerate}
\end{lmm}
\begin{proof}
We define $\mathbb{T}^{n}_{k, \sigma, \tau} \colon I^{n} \to I^{n}$ by 
$
\mathbb{T}^{n}_{k, \sigma, \tau}|K_{J_{k}}=
T^{n}_{\sigma, \tau} \colon I^{n} \cong K_{J_{k}} \to K_{J_{k}}
$
for each $J_{k}$.
Then it is smooth by tameness.
Moreover,
if we put
\[
H(t,s)=(1-s)t+s\mathbb{T}^{n}_{k,\sigma, \tau}(t),\ \mbox{for}\ (t,s) \in I^{n} \times I
\]
then $H \colon I^{n} \times I \to I^{n}$ gives $1_{I^{n}} \simeq \mathbb{T}^{n}_{k, \sigma, \tau}$.
\end{proof}
Let $A_{1}$ and $A_{2}$ be subspaces of a diffeologilcal space $X$.
We say that $(X;A_{1},A_{2})$ is a excisive triad if $X$ is the union of the interiors of $A$ and $B$.
We note that $A_{2}$ is not required to be a subspace of $A_{1}$.
Then we have the following.
\begin{thm}\label{thm:extends weak}
If $\phi \colon (X;A_{1},A_{2}) \to (Y;B_{1},B_{2})$ is a smooth map of excive triads such that the restrictions
\[
\phi \colon A_{1} \to B_{1},\ \phi \colon A_{2} \to B_{2} \ \mbox{and}\ \phi \colon A_{1} \cap A_{2} \to B_{1} \cap B_{2}
\]
are weak homotopy equivalences,
then $\phi \colon X \to Y$ is a weak homotopy equivalence.
\end{thm}
\begin{proof}
First we show that the induced map $\phi_{\ast} \colon \pi_{n}(X,x) \to \pi_{n}(Y,\phi(x))$ is injective for every $n \geq 0$ and $x \in X$.
Let $f_{0}$ and $f_{1}$ be smooth maps of pairs from $(I^{n} \partial I^{n})$ to $(X,x)$ such that there exists a tame homotopy
$g \colon \phi \circ f_{0} \simeq \phi \circ f_{1}$ relative to $\partial I^{n}$.
Let $F \colon \partial I^{n+1} \to X$ be a tame map given by the formula
\[
F(t,s)=\left\{
\begin{array}{llll}
f_{0}(t),
&
(t,s) \in I^{n} \times \{0\}
\\
x,
&
(t,s) \in \partial I^{n} \times I
\\
f_{1}(t),
&
(t,s) \in I^{n} \times \{1\},
\end{array}
\right.
\]
then $\phi \circ F=g|\partial I^{n+1}$ holds.
By taking $k$ large enough,
we can cubically subdivide $I^{n}$ into subcubes $K_{J_{k}}$ such that 
\[
F(\partial I^{n} \cap K_{J_{k}}) \subset {\rm Int}A_{i}\ \mbox{and} \ g(K_{J_{k}}) \subset {\rm Int}B_{i},
\]
where $i$ is either $1$ or $2$.
Let us put $F^{\prime}=F \circ \mathbb{T}^{n+1}_{k,\sigma,\tau}$ and 
$G^{\prime}=g \circ \mathbb{T}^{n+1}_{k,\sigma,\tau}$.
By Lemma \ref{lmm:locally tame map},
we have the following conditions.
\begin{enumerate}
\item
The restrictions $F^{\prime}|\partial I^{n+1} \cap K_{J_{k}}$ and $G^{\prime}|K_{J_{k}}$ are tame for all $J_{k}$.
\item
$\phi \circ F^{\prime}=G^{\prime}|\partial I^{n+1},\ F \simeq F^{\prime} \ \mbox{and} \ g \simeq G^{\prime}$.
\item\label{2}
$F^{\prime}(\partial I^{n+1} \cap K_{J_{k}}) \subset {\rm Int}A_{i}$ and $G^{\prime}(K_{J_{k}}) \subset {\rm Int}B_{i}$\
($i=1$ or $2$).
\end{enumerate}
Let $M_{i}$ be the union of all subcubes $K_{J_{k}}$ satisfying the condition (\ref{2}) above.
Then each $M_{i}$ is a subcomplex of $I^{n+1}$ and $I^{n+1}=M_{1} \cup M_{2}$ holds.
By Lemma \ref{lmm:N-equivalence},
we can obtain a smooth map such that the left triangle in the diagram
\[
\xymatrix{
\partial I^{n} \cap (M_{1} \cap M_{2})
\ar[r]^(0.62){F^{\prime}}
\ar[d]
&
A_{1} \cap A_{2}
\ar[d]^{\phi}
\\
M_{1} \cap M_{2}
\ar[r]_{G^{\prime}}
\ar@{.>}[ru]_{\exists \tilde{F}^{\prime}}
&
B_{1} \cap B_{2}
}
\]
commutes,
together with a tame homotopy 
$
\tilde{G}^{\prime} \colon (M_{1} \cap M_{2}) \times I \to B_{1} \cap B_{2}
$
such that
\[
G^{\prime} \simeq \phi \circ \tilde{F}^{\prime} \ \mbox{rel} \ \partial I^{n+1} \cap (M_{1} \cap M_{2}).
\]
By Theorem \ref{HELP},
for each $i$ there are a smooth map $\tilde{F}^{\prime}_{i}$ and a tame homotopy $\tilde{H}_{i}$ such that the following diagram commutes
\[
\xymatrix{
M_{1} \cap M_{2}
\ar[rr]^(0.45){j_{0}}
\ar[dd]
&
&
(M_{1} \cap M_{2}) \times I
\ar[ld]_{\tilde{G}^{\prime}}
\ar[dd]
&
&
M_{1} \cap M_{2}
\ar[ll]_(0.42){j_{1}}
\ar[ld]_{\tilde{F}^{\prime}}
\ar[dd]
\\
&
B_{i}
&
&
A_{i}
\ar[ll]_(0.4){\phi}
&
\\
M_{i}
\ar[ru]^{G^{\prime}}
\ar[rr]_(0.46){j_{0}}
&
&
M_{i} \times I
\ar@{.>}[lu]^{\exists \tilde{H}_{i}}
&
&
M_{i}.
\ar[ll]^(0.44){j_{1}}
\ar@{.>}[lu]^{\exists \tilde{F}^{\prime}_{i}}
}
\]
We define a smooth map $\tilde{f} \colon I^{n+1} \to X$ by $\tilde{f}|M_{i}=\tilde{F}^{\prime}_{i}$ for $i=1,2$.
Then there exists a tame homotopy $H \colon \partial I^{n+1} \times I \to X$ satisfying $H(t,0)=\tilde{f}(t)$ and $H(t,1)=F(t)$.
Let $H^{\prime} \colon L^{n+1} \to X$ be a tame map given by the formula
\[
H^{\prime}(t,s)=
\left\{
\begin{array}{lll}
H(t,s),
&
(t,s) \in \partial I^{n+1} \times I
\\
\tilde{F}^{\prime \prime}(t),
&
(t,s) \in I^{n+1} \times \{0\}.
\end{array}
\right.
\]
By Lemma \ref{extend of J},
we have a homotopy $\tilde{H} \colon f_{0} \simeq f_{1}$ relative to $\partial I^{n}$.
Hence $\phi_{\ast}$ is injective.

Next,
we show that $\phi_{\ast}$ is surjective.
Let $g \colon (I^{n} \partial I^{n}) \to (Y, \phi(x))$ be a smooth map of pairs and 
let $f \colon \partial I^{n} \to X$ be a constant map with $x \in X$.
Then $\phi \circ f=g|\partial I^{n}$ holds.
By the similar argument above,
we have a smooth map of pairs $\tilde{f} \colon (I^{n}, \partial I^{n}) \to (X,x)$ satisfying 
\[
g \simeq \phi \circ \tilde{f}\ \mbox{rel}\ \partial I^{n}
\]
since $f$ is constant map and $G^{\prime} \colon I^{n} \times I \to Y$ is a homotopy relative to $\partial I^{n}$ between $g$ and 
$g \circ \mathbb{T}^{n}_{k,\sigma,\tau}$.
Thus $\phi_{\ast}$ is surjective.
\end{proof}
In topological homotopy theory,
the Whitehead theorem states that if a continuous map $f$ between topological CW complexes induces isomorphisms on all homotopy groups,
then $f$ is a homotopy equivalence (cf.~{\cite[Theorem 6]{JHC}}).
To prove the Whitehead theorem on gathered CW complexes,
we need the following.
\begin{prp}\label{prp:main prp}
Let $\phi \colon (X,A) \to (Y,B)$ be a $N$-equivalence.
Let $(Z,C)$ be a pair of gathered CW complexes.
Then the induced map
\[
\phi_{\ast} \colon 
[
Z,C;X,A
]
\to
[
Z,C;Y,B
]
\]
is injective if $\mbox{\rm dim} (Z,C) <N$ and surjective if $\mbox{\rm dim} (Z,C) \leq N$.
\end{prp}
\begin{proof}
Suppose dim$(Z,C) \leq N$.
Let $g \colon (Z,C) \to (Y,B)$ be a smooth map of pairs.
Let us denote by $g_{0}$ the restriction $g|A$.
By Lemma \ref{lmm:N-equivalence},
there are a smooth map $f_{0} \colon C \to A$ and a tame homotopy $G_{0} \colon C \times I \to B$ between $g_{0}$ and $\phi_{0} \circ f_{0}$,
where $\phi_{0}=\phi|A$.
Then we have a smooth map of pairs $f \colon (Z,C) \to (X,A)$ and a tame homotopy of pairs $G \colon Z \times I \to Y$ between $g$ and $\phi_{0} \circ f$
by Theorem \ref{HELP}.
Thus $\phi_{\ast}$ is a surjection.

Suppose dim$(Z,C)<N$.
Let $f, f^{\prime} \colon (Z,C) \to (X,A)$ be smooth maps of pairs such that there is a tame homotopy of pairs $G \colon X \times I \to Y$
between $\phi \circ f$ and $\phi \circ f^{\prime}$.
Let us denote by $G_{0}$ the restriction $G|C\times I$.
We define $K_{0} \colon C \times \partial I \to A$ by putting $K_{0}(z,0)=f(z)$ and $K_{0}(z,1)=f^{\prime}(z)$.
By Lemma \ref{lmm:N-equivalence},
there are tame homotopys $F_{0} \colon C \times I \to A$ and $G^{\prime}_{0} \colon C \times I \times I \to B$ between 
$G_{0}$ and $\phi_{0} \circ F_{0}$ relative to $C \times \partial I$.
By Lemma \ref{lmm:HEP} there exists an $\epsilon$-tame homotopy $\tilde{G}^{\prime} \colon Z \times I \times I \to Y$ satisfying $\tilde{G}^{\prime}(z,t,0)=G(z,t)$ and
$\tilde{G}^{\prime}|C \times I \times I=G^{\prime}_{0}$.
Let $\tilde{G}^{\prime \prime} \colon Z \times I \to Y$ be a tame homotopy of pairs between $\phi \circ f$ and $\phi \circ f^{\prime}$ 
given by the formula
\[
\tilde{G}^{\prime \prime}(z,t)=
\left\{
\begin{array}{ll}
\tilde{G}^{\prime}\left(z,0, {2t}/{\epsilon}\right), & 0 \leq t \leq {\epsilon}/{2} \\
\tilde{G}^{\prime}(z,t,1), & {\epsilon}/{2} \leq t \leq ({2-\epsilon})/{2} \\
\tilde{G}^{\prime}\left(z,1,({2-2t}){\epsilon}\right), & ({2-\epsilon})/{2} \leq t \leq 1,
\end{array}
\right.
\]
then $\tilde{G}^{\prime \prime}|C \times I=\phi_{0} \circ F_{0}$ holds.
We define $\tilde{F}_{0} \colon Z \times \partial I \cup C \times I \to X$ by putting
\[
\tilde{F}_{0}(z,0)=f(z), \
\tilde{F}_{0}(z,1)=f^{\prime}(z) \ \mbox{and}\ \tilde{F}_{0}|C\times I=F_{0}.
\]
Since $Z \times \partial I \cup C \times I$ is a subcomplex of a gathered CW complex $Z \times I$ 
by Proposition \ref{prp:example of CW complex} and Proposition \ref{prp:product of CW},
there is a tame homotopy of pairs $F \colon Z \times I \to X$ between $f$ and $f^{\prime}$
by Lemma \ref{lmm:N-equivalence}.
Thus $\phi_{\ast}$ is injection.
\end{proof}
\begin{thm}[The Whitehead theorem]\label{thm:Whitehead}
A $N$-equivalence $\phi \colon (X,A) \to (Y,B)$ between pairs of gathered CW complexes of dimension less than $N$ is a homotopy equivalence.
A weak homotopy euivalence between pairs of gathered CW complexes.
\end{thm}
\begin{proof}
By Proposition \ref{prp:main prp},
the induced map
\[
\phi_{\ast} \colon 
[Y, B; X,A] \to [Y,B;Y,B]
\]
is surjective.
Thus there exists a smooth map of pairs $\psi \colon (Y,B) \to (X,A)$ satisfying  $\phi \circ \psi \simeq 1_{Y}$.
Since $\phi$ is a $N$-equivalence,
so is $\psi$.
Hence the induced map 
\[
\psi_{\ast} \colon 
[X,A;Y,B] \to [X,A;X,A]
\]
is surjective by Proposition \ref{prp:main prp}.
Then there is a smooth map of pairs $\psi \colon (X,A) \to (Y,B)$ satisfying $\psi \circ \phi^{\prime} \simeq 1_{X}$.
We have
\[
\phi^{\prime} \simeq 1_{Y} \circ \phi^{\prime} \simeq \phi \circ \psi \circ \phi^{\prime} \simeq \phi \circ 1_{X} \simeq \phi.
\]
Hence $\phi$ is a homotopy equivalence since $\psi \circ \phi \simeq 1_{X}$ and $\phi \circ \psi \simeq 1_{Y}$ hold.
\end{proof}
%
%
%
%
\section{The cellular approximation theorem}
In this section we prove the cellular approximation theorem on relative gathered CW complexes.
\begin{dfn}\label{dfn:cellular map}
Let $(X,A)$ and $(Y,B)$ be smooth relative CW complexes.
Then a smooth map of pairs $f \colon (X,A) \to (Y,B)$ is said to be cellular if $f(X^{n}) \subset Y^{n}$ for all $n$.
\end{dfn}
\begin{thm}[Cellular approximation theorem]\label{thm:cellular approximation}
Let $f \colon (X,A) \to (Y,B)$ be a smooth map between relative gathered CW complexes.
Then there exists a cellular map $g \colon (X,A) \to (Y,B)$ such that $f$ and $g$ are homotopic relative to $A$.
\end{thm}
To prove theorem above,
we need the followings.
\begin{lmm}\label{retraction lemma}
Let $(X,A)$ be a relative gathered CW complex.
For each $n$,
there exists a $D$-open set $B^{n}$ of $X^{n}$ such that $X^{n-1}$ is a deformation retract of $B^{n}$.
\end{lmm}
\begin{proof}
Suppose $X^{n}$ is given by the following pushout
\[
\xymatrix{
\coprod_{\lambda \in \Lambda_{n}} \partial I^{n}
\ar[r]^(0.6){\bigcup_{\lambda} \phi_{\lambda}}
\ar[d]
&
X^{n-1}
\ar[d]^{i_{n}}
\\
\coprod_{\lambda \in \Lambda_{n}} I^{n}
\ar[r]_(0.6){\bigcup_{\lambda} \Phi_{\lambda}}
&
X^{n},
}
\]
where $\phi_{\lambda}$ is $\tau_{\lambda}$-tame homotopy.
Let us define a tame homotopy $F_{\lambda} \colon I^{n} \times I \to X^{n}$ relative to $\partial I^{n}$ by
\[
F_{\lambda}(t,s)=
\Phi_{\lambda}((1-T_{\epsilon, 1/2}(s))t+T_{\epsilon, 1/2}(s)T^{n}_{\sigma_{\lambda},\tau_{\lambda}}(t)).
\]
By Lemma \ref{subduction of final diffeology},
we have an $\epsilon$-tame homotopy $\tilde{F}_{\lambda} \colon X^{n} \times I \to X^{n}$ relative to $X^{n-1}$ such that 
the following diagram commutes
\[
\xymatrix@C=45pt{
(X^{n-1} \times I) \coprod_{\lambda \in \Lambda_{n}} (I^{n} \times I)
\ar[r]^(0.75){i \bigcup F_{\lambda}}
\ar[d]_{(i_{n} \times 1_{I}) \bigcup_{\lambda} (\Phi_{\lambda} \times 1_{I})}
&
X^{n}
\\
X^{n} \times I,
\ar@{.>}[ru]_{\exists \tilde{F}_{\lambda}}
}
\]
whrere $i \colon X^{n-1} \times I \to X^{n}$ is given by $i(x,t)=x$.
Let $B^{n}$ be the $D$-open set of $X^{n}$ given by the adjunction space $X^{n-1} \cup_{\phi_{\lambda}} K_{\lambda}$,
where $K_{\lambda}=I^{n} \setminus I^{n}(\sigma_{\lambda})$  
Then we have a retraction $\gamma^{n} \colon B^{n} \to X^{n-1}$ given by $\gamma^{n}(x)=\tilde{F}_{\lambda}(x,1)$.
Clearly,
$\gamma^{n} \circ i_{n}=1_{X^{n-1}}$ and $i_{n} \circ \gamma^{n} \simeq 1_{B^{n}}$ hold.
\end{proof}
\begin{lmm}\label{lmm:cellular of cube}
Let $(X^{n},A)$ be a $n$-dimensional relative gathered CW complex.
Let $n>m$ and let $f \colon I^{m} \to X^{n}$ be a smooth map such that restriction $f|\partial I^{m}$ is tame and $f(\partial I^{m}) \subset X^{m-1}$ holds.
Then there exists a smooth map $f^{\prime} \colon I^{m} \to X^{n}$ satisfying $f \simeq f^{\prime}$ rel $\partial I^{m}$ and $f^{\prime}(I^{m}) \subset X^{n-1}$.
\end{lmm}
\begin{proof}
By Lemma \ref{retraction lemma},
there exists a $D$-open set $B^{n}$ of $X^{n}$ such that $X^{n-1}$ is a deformation retract of $B^{n}$.
Then $\{f^{-1}(B^{n}),f^{-1}(X^{n} \setminus X^{n-1}) \}$ is an open cover of $I^{m}$.
By taking $k$ large enough,
we can cubically subdivide $I^{n}$ into subcubes $K_{J_{k}}$ such that $f(K_{J_{k}})$ is contained in either $B^{n}$ or $X^{n} \setminus X^{n-1}$.
Let us denote by $f_{1}$ the composition $f \circ \mathbb{T}^{n}_{k, \sigma, \tau}$.
Then we have the following conditions.
\begin{enumerate}
\item
The restriction $f_{1}|K_{J_{k}}$ is tame for each $J_{k}$.
\item
$f$ and $f_{1}$ are homotopic.
\item
$f_{1}(K_{J_{k}})$ is contained in either $B^{n}$ or $X^{n} \setminus X^{n-1}$.
\end{enumerate}
Let $M_{1}$ and $M_{2}$ be the unions of subcubes 
$K_{J_{k}}$ such that $f_{1}(K_{J_{k}}) \subset B^{n}$ and $f_{1}(K_{J_{k}}) \subset X^{n} \setminus X^{n-1}$,
respectively.
Then $M_{1}$ and $M_{2}$ are subcomplexes of $I^{m}$ such that $M_{1} \cup M_{2}=I^{m}$
For each $N \leq n-1$ and for each $\lambda \in \Lambda_{n}$,
we have
\[
\pi_{N}(\Phi_{\lambda}(\mbox{Int}I^{n}), B^{n} \cap \Phi_{\lambda}(\mbox{Int}I^{n})) \cong \pi_{N}(D^{n},S^{n-1}) \cong \pi_{N-1}(S^{n-1})=0,
\]
where $\Phi_{\lambda} \colon I^{n} \to X^{n}$ is a characteristic map.
Hence a pair $$(X^{n} \setminus X^{n-1}, B^{n} \cap (X^{n} \setminus X^{n-1}))$$ is $(n-1)$-connected.
By Proposition \ref{prp:homotopic},
there are a smooth map of pairs
\[
f_{2} \colon (M_{2},M_{1} \cap M_{2}) \to (X^{n} \setminus X^{n-1},B^{n} \cap (X^{n} \setminus X^{n-1}))
\] 
and a tame homotopy of pairs $F \colon M_{2} \times I \to X^{n} \setminus X^{n-1}$ relative to $M_{1} \cap M_{2}$
such that $f_{2}(M_{2}) \subset B^{n} \cap (X^{n} \setminus X^{n-1})$,
$F(t,0)=f_{1}(t)$ and $F(t,1)=f_{2}(t)$ hold.
Let $F^{\prime} \colon I^{m} \times I \to X^{n}$ be a tame homotopy defined by
\[
F^{\prime}(t,s)=
\left\{
\begin{array}{lll}
f_{1}(t), & (t,s) \in M_{1} \times I \\
F(t,s), & (t,s) \in M_{2} \times I.
\end{array}
\right.
\]
and let $f_{3} \colon I^{m} \to X^{n}$ be a smooth map given by $f_{3}(t)=F^{\prime}(t,1)$.
Since $f_{3}(I^{m}) \subset B^{n}$ holds,
we have
$
f \simeq f_{1} \simeq f_{3} \simeq \gamma^{n} \circ f_{3}
$
by Lemma \ref{retraction lemma}.
Hence we have a smooth map $f^{\prime} \colon I^{m} \to X^{n}$ satisfying $f^{\prime}(I^{m}) \subset X^{n-1}$ and 
$f \simeq f^{\prime}$ rel $\partial I^{n}$ by Theorem \ref{HEP2}.
\end{proof}
We are now ready to prove the cellular approximation theorem. 
\begin{proof}[\it Proof of Theorem \ref{thm:cellular approximation}]
Since $X^{0} \setminus A$ is discrete,
we can construct a tame homotopy $F^{0} \colon X^{0} \times I \to Y$ relative to $A$ satisfying $F^{0}(x,0)=f(x)$ and 
$F^{0}(x,1) \in Y^{0}$ by Lemma \ref{lmm:cellular of cube}.
Suppose we shall inductively construct a cellular map $g^{n-1} \to Y$ and a tame homotopy $F^{n-1} \colon f|X^{n-1} \simeq g^{n-1}$ 
satisfying $F^{n-1}|X^{n-2} \times I= F^{n-2}$.
For each $\lambda \in \Lambda_{n}$,
let $\phi_{\lambda} \colon \partial I^{n} \to X^{n-1}$ and $\Phi_{\lambda} \colon I^{n} \to X^{n}$ be attaching and characteristic maps,
respectively.
Let us define a smooth map $F_{\lambda} \colon L^{n} \to Y$ by putting
\[
F_{\lambda}(t,s)=
\left\{
\begin{array}{lll}
F^{n-1}(\phi_{\lambda}(t),s),
&
(t,s) \in \partial I^{n} \times I
\\
f(\Phi_{\lambda}(t)),
&
(t,s) \in I^{n} \times \{0\}.
\end{array}
\right.
\]
We have a $\tau$-tame homotopy $F^{\prime}_{\lambda} \colon I^{n} \times I \to Y$ extending $F_{\lambda}$ 
by Proposition \ref{prp:smooth retraction}.
Let $g_{\lambda} \colon (I^{n}, \partial I^{n}) \to (Y, Y^{n-1})$ be a smooth map of pairs given by $g_{\lambda}(t)=F^{\prime}_{\lambda}(t,1)$ for $t \in I^{n}$.
Since it satisfies $g_{\lambda}(I^{n}) \subset Y^{m}$ for some $m \geq n-1$,
there are a smooth map of pairs $g^{\prime}_{\lambda} \colon (I^{n}, \partial I^{n}) \to (Y,Y^{n-1})$ satisfying $g^{\prime}_{\lambda}(I^{n}) \subset Y^{n-1}$
and a tame homotopy of pairs $G_{\lambda} \colon g_{\lambda} \simeq g^{\prime}_{\lambda}$ relative to $\partial I^{n}$ by Lemma \ref{lmm:cellular of cube}.
A tame homotopy $\tilde{F}_{\lambda} \colon I^{n} \times I \to Y$ given by the formula
\[
\tilde{F}_{\lambda}(t,s)=
\left\{
\begin{array}{lll}
F^{\prime}_{\lambda}(t,s),
&
0 \leq s \leq ({2-\tau})/{2}
\\
G_{\lambda}(t,(2s-2+\tau)/\tau),
&
({2-\tau})/{2} \leq s \leq 1
\end{array}
\right.
\]
satisfies $\tilde{F}_{\lambda}(t,1)=g^{\prime}_{\lambda}(t)$ and $\tilde{F}_{\lambda}|L^{n}=F_{\lambda}$.
By Lemma \ref{subduction of final diffeology},
there exists a tame homotopy $F^{n} \colon X^{n} \times I \to Y$ such that the following diagram commutes
\[
\xymatrix@C=60pt{%
(X^{n-1} \times I)\coprod_{\lambda \in \Lambda_{n}} (I^{n} \times I)
\ar[r]^-{F^{n-1} \bigcup_{\lambda} \tilde{F}_{\lambda}} 
\ar[d]_-{(i_{n} \times 1_{I}) \bigcup_{\lambda}( \Phi_{\lambda} \times 1_{I}) } & Y
\\
X^{n} \times I
\ar@{.>}[ru]_-{\exists {F}^{n}}. }%
\]
Hence we have a cellular map of pairs $g \colon (X,A) \to (Y,B)$ and a tame homotopy $F \colon X \times I \to Y$ between $f$ and $g$ relative to $A$ given by 
$g|X^{n}=g^{n}$ and $F|X^{n} \times I=F^{n}$,
respectively.
\end{proof}
%
%
%
%
%
%

\end{document}